\documentclass{amsart}
\usepackage{graphicx} 
\usepackage[T2A]{fontenc}
\usepackage[utf8]{inputenc}
\usepackage{dsfont}
\usepackage{amsfonts}
\usepackage{amssymb}
\usepackage{amsmath}
\usepackage{amsthm}
\usepackage{graphicx}
\usepackage{todonotes}
\usepackage{tikzsymbols}

\usepackage{bm}

\def \le {\leqslant}
\def \ge {\geqslant}
\makeatletter
\DeclareRobustCommand{\cev}[1]{%
  \mathpalette\do@cev{#1}%
}
\newcommand{\do@cev}[2]{%
  \fix@cev{#1}{+}%
  \reflectbox{$\m@th#1\vec{\reflectbox{$\fix@cev{#1}{-}\m@th#1#2\fix@cev{#1}{+}$}}$}%
  \fix@cev{#1}{-}%
}
\newcommand{\fix@cev}[2]{%
  \ifx#1\displaystyle
    \mkern#23mu
  \else
    \ifx#1\textstyle
      \mkern#23mu
    \else
      \ifx#1\scriptstyle
        \mkern#22mu
      \else
        \mkern#22mu
      \fi
    \fi
  \fi
}

\newcommand{\cc}{\bm{c}}

\newtheorem{theor}{Theorem}
\newtheorem{lem}[theor]{Lemma}
\newtheorem{prop}[theor]{Proposition} 
\newtheorem{defprop}[theor]{Definition/Proposition} 

\theoremstyle{remark}
\newtheorem*{rmk*}{Remark}

\DeclareMathOperator{\N}{\mathbb{N}}
\DeclareMathOperator{\Q}{\mathbb{Q}}

\newcommand{\dk}{10^{-100}}

\author{Dmitry Gayfulin}
        \address{\parbox{\linewidth} {Institute of Analysis and Number Theory, TU Graz, Steyrergasse 30, 8010 Graz, Austria\\
        and\\
        Institute for Information Transmission Problems,
        19 Bolshoy Karetnyi side-str., Moscow 127994, Russia.}
        }
	\email{gamak.57.msk@gmail.com}
\author{Manuel Hauke}
	\address{\parbox{\linewidth} {Department of Mathematics, University of York, 
 York YO10 5DD,
United Kingdom}}
	\email{manuel.hauke@york.ac.uk}
\title{Hausdorff dimension estimates for Sudler products with positive lower bound}
\begin{document}

\maketitle

\begin{abstract}
Given an irrational number $\alpha$, we study the asymptotic behaviour of
the Sudler product denoted by $P_N(\alpha) =
\prod_{r=1}^N  2\lvert \sin \pi r \alpha \rvert$. We show that $\liminf_{N \to \infty} P_N(\alpha) >0$
and $\limsup_{N \to \infty} P_N(\alpha)/N < \infty$ whenever the sequence of partial quotients in the continued
fraction expansion of $\alpha$ exceeds 3 only finitely often, which confirms a conjecture of the second-named author and partially answers a question of J. Shallit. Furthermore, we show that the Hausdorff dimension of the set of those $\alpha$ that satisfy $\limsup_{N \to \infty} P_N(\alpha)/N < \infty,\liminf_{N \to \infty} P_N(\alpha) >0$ lies between $0.7056$
and $0.8677$, which makes significant progress on a question raised by Aistleitner, Technau, and Zafeiropoulos.
\end{abstract}

\section{Introduction and statement of results}

For $\alpha \in \mathbb{R}$ and $N \in \N$, the Sudler product at stage $N$ is defined as
\begin{equation*}
P_N(\alpha) := \prod_{r=1}^{N} 2 \left\lvert \sin \pi r\alpha \right\rvert.
\end{equation*}
Sudler products appear in many different areas of mathematics that include, among others, KAM theory \cite{kam}, restricted partition functions \cite{sudler}, Padé approximants \cite{lubinsky_pade}, almost Mathieu operators \cite{aj, koch}, and invariants of hyperbolic knots 
\cite{quantum_invariants,zag_conj,bd1,bd2,bd3}. \par

Writing $\| P_N \|_{\infty} = \sup\limits_{\alpha \in \mathbb{R}} P_N(\alpha)$, Erd\H os and Szekeres \cite{erdos_szekeres} claimed that the limit
$\lim_{N \to \infty} \| P_N \|_{\infty}^{1/N}$ exists and equals a value between $1$ and $2$, without formally proving it. This was done by Sudler \cite{sudler} and Wright \cite{wright} who showed that
$\lim_{N \to \infty} \| P_N \|_{\infty}^{1/N} = C \approx 1.22$.
Inspired by this, the order of growth of Sudler products was extensively examined from a metric point of view. For more results in this area, we refer
the reader to \cite{atkinson,bbr,bell,bglly,bc,fh,kol,kol2}.\\

The object of interest in this article is the pointwise behaviour of the Sudler product, which has also gained a lot of interest in recent years. Here, $\alpha \in \mathbb{R}$ is fixed and the behaviour of 
$P_N(\alpha)$ when $N \to \infty$ is examined. Observe that by $1$-periodicity and the fact that for $\alpha = \frac{p}{q} \in \Q$ we have $P_N(\alpha) = 0$ for $N \ge q$, the asymptotic analysis can be restricted to irrational numbers $\alpha \in [0,1)$.
A central object in this setting is given by the question of characterizing the set
\[S := \left\{\alpha \in [0,1): \liminf_{N \to \infty} P_N(\alpha) >0\right\}.\]
First studied by Erd\H os and Szekeres \cite{erdos_szekeres}, 
it was proven that $S$ forms a set of Lebesgue measure $0$. This was improved by Lubinksy \cite{lubinsky} who was the first to develop a connection between the asymptotic behaviour of the Sudler product and the Diophantine properties of $\alpha$. Thereby, he showed a certain rate of convergence of $\liminf_{N \to \infty} P_N(\alpha)$ to $0$ for almost every $\alpha$; for further results for almost all irrationals see e.g. \cite{borda,hauke_density}. Furthermore, Lubinsky proved that $S$ is contained in the set of badly approximable numbers.
By a careful analysis of Lubinsky's proof, it was recovered  in \cite{grepstad_survey} that there is actually a finite cutoff value $K \approx e^{800}$ such that (when writing $\alpha = [0;a_1,a_2,\ldots]$ in its continued fraction expansion),
 \[S \subseteq \tilde{E}_K := \{\alpha \in [0,1): \limsup_{i \to \infty} a_i(\alpha) \le K\}.\]
Lubinsky conjectured that $S$ is actually empty, which was disproven by Verschueren \cite{versch} and independently Grepstad, Kaltenböck and Neumüller \cite{grepstad_neum}, as they 
proved that the Golden Ratio $\phi = [0;1,1,1,\ldots] \in S$.\par
This result was generalized by Aistleitner, Technau, and Zafeiropoulos \cite{tech_zaf} who showed that for $\beta(b) := [0;b,b,b,\ldots]$, we have
$\beta(b) \in S \Leftrightarrow b \le 5.$
At the end of the article, they raised the question of characterizing those irrationals whose Sudler product grows at most linearly. This is known to be equivalent to characterizing $S$ since it can be deduced from the work of Aistleitner and Borda \cite{quantum_invariants} that for any irrational $\alpha$,

\begin{equation}\label{liminf_limsup_equiv}\liminf_{N \to \infty} P_N(\alpha) = 0 \; \Longleftrightarrow \; \limsup_{N \to \infty} \frac{P_N(\alpha)}{N} = \infty.\end{equation}
\par 

In recent articles, significant progress was made in providing sufficient conditions under which an irrational does \textit{not} lie in $S$. Grepstad, Neumüller and Zafeiropoulus \cite{grepstadII} showed that $\alpha \notin S$ if $\alpha$ is a quadratic irrational with $\limsup_{n \to \infty} a_n \ge 23$, conjecturing that already 
$\limsup_{n \to \infty} a_n \ge 6$ suffices. This turned out to be false in a recent paper of the second author \cite{hauke_badly}: generalizing the method of proof to arbitrary badly approximable numbers, the value $23$ was decreased to the optimal bound, showing that
\begin{equation}\label{hauke_estim}
S \subseteq \tilde{E}_6, \quad S \not\subseteq \tilde{E}_5,
\end{equation}
where
\[\tilde{E}_k:= \{\alpha \in [0,1): \limsup_{i \to \infty} a_i(\alpha) \le k\}.\]

The current article aims to enlarge the set of known irrationals $\alpha$ where 
$\liminf_{N \to \infty} P_N(\alpha) > 0$ holds. Note that up to now, there is no single non-quadratic irrational known to lie in $S$. This is due to the fact that all proofs given so far heavily use the existence of periodicity in the continued fraction expansion of quadratic irrationals. However, based on numerical evidence, it was conjectured in \cite{hauke_badly} that 
$\tilde{E}_3 \subseteq S$. As a main result of this paper, we confirm this conjecture:

\begin{theor}
\label{thm_main}
Let $\alpha$ be an irrational number with its continued fraction expansion 
$\alpha=[0;a_1,a_2,\ldots]
$ satisfying $\limsup_{i \to \infty} a_i \le 3$.
Then we have
$$
\liminf\limits_{N\to\infty}P_N(\alpha)>0, \quad
\limsup_{N \to \infty} \frac{P_N(\alpha)}{N} < \infty.
$$
\end{theor}

As an immediate corollary of Theorem \ref{thm_main}, we obtain non-trivial (upper and lower) Hausdorff dimension estimates for the set $S$:

\begin{theor}
\label{thm_hausdorff}
Let $S :=
\left\{\alpha \in [0,1): \liminf\limits_{N\to\infty}P_N(\alpha)>0, \limsup_{N \to \infty} \frac{P_N(\alpha)}{N} < \infty\right\}.
$
Then we have
\[
0.7056 \le \dim_H(S) \le 0.8677.
\]
\end{theor}

\begin{proof}[Proof of Theorem \ref{thm_hausdorff}, assuming Theorem \ref{thm_main}]
By Theorem \ref{thm_main} and \eqref{hauke_estim}, we have 
$\tilde{E}_3 \subseteq S \subseteq \tilde{E}_6$.
By standard arguments (taking a countable union of sets of the same Hausdorff dimension), we obtain
$\dim_H(\tilde{E}_k) =  \dim_H(E_k)$ 
where 
\[E_k := \{\alpha \in [0,1): \max_{i \in \N} a_i(\alpha) \le k\}.\]
Using estimates on $\dim_H(E_k)$ (see e.g. \cite[Table 1]{jenk} and \cite[Theorem 4.22]{povy}), we obtain
\[0.7056 \le \dim_H(E_3) \le \dim_H(S) \le  \dim_H(\tilde{E}_6) \le 0.8677.\]
\end{proof}

Here and in preceding results on the characterization of $S$, it seems that the question of $\alpha \in S$ does not depend on its first finitely many partial quotients, or in other words,
that $S$ is invariant under the Gauss map $T(x) :=\left\{\frac{1}{x}\right\}$.
However, to the best knowledge of the authors, this was never proven explicitly, thus we provide a proof of this fact.

\begin{theor}
\label{thm_Gauss}
Let $\alpha\in \mathbb{R}$ and $n \in \mathbb{N}$ arbitrary. Then 
\[\alpha \in S \Leftrightarrow T^n(\alpha)\in S.\]
\end{theor}

\subsection*{Remarks on the Theorems and directions for further research}

\begin{itemize}
\item Since all the elements in $S$ known so far were quadratic irrationals, $S$ was not even known to be uncountable. 
Actually, by considering the (in view of Theorem \ref{thm_Gauss}) natural equivalence relation given by $\alpha \sim \beta \iff 
T^{m}(\alpha) = T^n(\beta), m,n \in \mathbb{N}$,
there were only seven examples (see \cite{tech_zaf,hauke_badly}) known to lie in $S$.
Hence, Theorem \ref{thm_hausdorff} is a significant improvement to the question raised in \cite{tech_zaf} by not only showing uncountability but giving nontrivial estimates on the Hausdorff dimension.
\item In terms of the value of $K(\alpha):= \limsup_{i \to \infty} a_i$, Theorem \ref{thm_main} provides a sharp threshold, thus completing the best possible description of $\liminf_{N \to \infty} P_N(\alpha)$ in terms of $K(\alpha)$: If $K(\alpha) \le 3$, we have $\liminf_{N \to \infty} P_N(\alpha) > 0$, if $K(\alpha) \ge 7$, we have $\liminf_{N \to \infty} P_N(\alpha) = 0$. If
$K(\alpha) \in \{4,5,6\}$, the value of $K(\alpha)$ does not determine whether $\alpha \in S$: for each value $4,5,6$, there exist irrationals $\alpha,\beta$ with $K(\alpha) = K(\beta)$ with $\liminf_{N \to \infty} P_N(\alpha) > 0$ and $\liminf_{N \to \infty} P_N(\beta) = 0$ (for examples, see \cite{tech_zaf} and \cite{hauke_badly}) .
\item For integers $u \ge 2$, the numbers $\alpha_u := \sum_{k = 0}^{\infty} \frac{1}{u^{2^k-1}}$
are known to be transcendental, with their partial quotients bounded by $u$.
These numbers are known as examples of badly approximable numbers that are not quadratic irrationals,
but their corresponding continued fraction expansions can still be described by a relatively simple pattern, see \cite{shallit}.
The second-named author was asked by J. Shallit in private communication whether 
something can be said about the question for which $u$, $\alpha_u \in S$. By Theorem \ref{thm_main} we know that $\alpha_2,\alpha_3 \in S$ and by \eqref{hauke_estim}, $\alpha_u \notin S$ for $u \ge 7$. Again, the question of whether $\alpha_u \in S$ for $u \in \{4,5,6\}$ remains open, but the question might be easier to approach than proving statements for arbitrary non-quadratic badly approximable numbers since the continued fraction expansion follows a particular pattern.
\item Most probably, both the upper and lower bounds for the Hausdorff dimension can be improved by considering different criteria when $K(\alpha) \in \{4,5,6\}$. While some patterns like $(1,6,1)$ appearing infinitely often in the continued fraction expansion probably lead to $\liminf_{N \to \infty} P_N(\alpha) = 0$, the approach considered in this paper has its limitations when it comes to determining the \textit{precise} Hausdorff dimension, let alone complete characterization for the set $S$. To determine this, we expect a genuinely new approach to be needed. 
\item In this paper we prove a \textit{pointwise} lower bound for $P_N(\alpha)$ when $\alpha \in \tilde{E}_3$. Having particularly large numbers in the first finitely many partial quotients seems to decrease the value of $\liminf_{N \to \infty} P_N(\alpha)$ more and more, hence most probably,
\[
\inf_{\alpha \in \tilde{E}_3}\liminf_{N \to \infty} P_N(\alpha) = 0.
\]
However, 
we actually show a \textit{uniform} lower bound for $E_3$, i.e.
\begin{equation}
    \label{uniform_bound}
\inf_{\alpha \in E_3}\liminf_{N \to \infty} P_N(\alpha) = c > 0.
\end{equation}

It would be interesting to determine the value of $c$ or to characterize a minimizing
subsequence along which the $\liminf$ is obtained for fixed $\alpha$. In \cite{hauke_extreme}, the actual value of $\liminf_{N \to \infty} P_N([0;1,1,1,\ldots])$ was determined, 
with a minimizing subsequence being the Fibonacci numbers, that is, in the corresponding Ostrowski expansion (see Section \ref{prereq}), we only have one non-zero Ostrowski digit. The method of proof reveals that for any $\alpha \in E_3$, every minimizing subsequence can only have boundedly many non-zero Ostrowski digits, but the particular structure of a minimizing subsequence remains unclear in the more complicated setting of arbitrary badly approximable numbers. 

\item From a technical perspective, 
there are two obstacles to circumvent in order to prove Theorem \ref{thm_main}. Firstly, showing $\liminf_{N \to \infty} P_N(\alpha) = 0$ is much easier than proving $\liminf_{N \to \infty} P_N(\alpha) > 0$ since it suffices to consider a suitable subsequence, instead of proving a uniform positive lower bound for \textit{all} indices. For this, our method follows a refined version of the ideas established in \cite{tech_zaf} for quadratic irrationals.
Secondly, analyzing arbitrary badly approximable numbers instead of quadratic irrationals is a tougher challenge since there is no periodicity to exploit. For this, we borrow ideas from \cite{hauke_badly}. The proof of Theorem \ref{thm_main} combines these two methods, together with new ideas and an extensive case distinction.
\end{itemize}

\subsection*{Structure of the paper}
The rest of the article is structured as follows. In Section \ref{prereq}, we recall elementary facts 
as well as results from various preceding articles on Sudler products that will be applied later. In Section \ref{main_ideas}, we present the main idea, and prove how Theorem \ref{thm_main} can be deduced from the two main lemmas in this article (Lemmas \ref{new_idea_lem} and \ref{main_techn_lem_variant}). In Section \ref{Hk_section}, we construct a family of functions that plays a key role in this paper and prove Lemma \ref{new_idea_lem}. Finally, in Section \ref{case_study}, we compute various estimates for the functions constructed in Section \ref{Hk_section} and prove Lemma \ref{main_techn_lem_variant}. The proof of Theorem \ref{thm_Gauss} is quite independent from our main argument, and therefore we put it into the Appendix.

\section{Some notation and prerequisites}\label{prereq}
\subsection*{Notation}
Throughout the text, we use the standard O- and o- notations and write
$f(n) \ll g(n), f(n) \gg g(n)$ if $f(n) = O(g(n))$ and $g(n) = O(f(n))$ respectively.
If $f(n) \ll g(n) \ll f(n)$, we write
$f(n) \asymp g(n)$.
Given a real number $x\in \mathbb{R},$ we write $\{x\} := x - \lfloor x \rfloor$ for the fractional part of $x$, and $\|x\| :=\min\{|x-k|: k\in\mathbb{N}\}$ for the distance of $x$ to the nearest integer.

\subsection*{Concavity properties}
In the subsequent proofs, we will make use of several elementary properties.
Recall that a function $f$ is called \textit{log-concave} when $\log(f)$ is concave. Further, a function is called \textit{pseudo-concave} on $[a,b]$ when for every interval
$[c,d] \subseteq [a,b]$ and any $c \le x \le d$ we have
\begin{equation*}
    f(x) \ge \min\{f(c),f(d)\}.
\end{equation*}
We recall the following well-known facts:

\begin{itemize}
    \item If $f$ is concave and strictly positive, then $f$ is log-concave.
    \item If $f$ is log-concave, then $f$ is pseudo-concave.
    \item Products of log-concave functions are log-concave.
\end{itemize}

\subsection*{Continued fractions and Ostrowski representation}
Here we recall some basic
facts about continued fractions that we use in this article. For a more detailed background, we refer to classical literature e.g. \cite{all_shall,khinchin, rock_sz,schmidt}.
Any real number $\alpha$ can be represented as a continued fraction expansion
\begin{equation}
\label{xexp}
\alpha=[a_0;a_1,a_2,\ldots]=a_0+ \cfrac{1}{a_1+\cfrac{1}{a_2+\cdots}},
\end{equation}
where $a_0\in\mathbb{Z}$ and $a_i\in\mathbb{Z}_{+}$ for $i\ge 1$. This representation is finite for rational $\alpha$, and infinite and unique for irrational $\alpha$.
We call an irrational number $\alpha$ \textit{badly approximable} if $\limsup_{n \to \infty} a_n(\alpha) < \infty$. 
The convergents of a continued fraction are the rational numbers 
$$
\frac{p_k}{q_k}:=[a_0;a_1,a_2,\ldots,a_k]
$$
whose denominators and numerators satisfy 
$p_0 = a_0,\; p_1 = a_1a_0 +1,\; q_0 = 1,\; q_1 = a_1$ and for $k \ge 1$ the recurrence law
$$
p_{k+1}(\alpha)=a_{k+1}(\alpha)p_k(\alpha)+p_{k-1}(\alpha),\ \ q_{k+1}(\alpha)=a_{k+1}(\alpha)q_k(\alpha)+q_{k-1}(\alpha).
$$
Throughout this article, we will omit the dependency of $a_k$, $p_k$, and $q_k$ on $\alpha$ when it does not create ambiguity. Defining 
\begin{equation}
\label{alpha_delta_def}
    \delta_k := \| q_k \alpha\|,\quad 
\alpha_{k} := [a_{k};a_{k+1},a_{k+2},\ldots], \quad
\cev{\alpha}_k := [0;a_k,a_{k-1},a_{k-2},\ldots,a_1],\quad \lambda_k:=q_k\delta_k, \quad k \ge 1,
\end{equation}
 we straightforwardly obtain 

\begin{align}
\label{alternating_approx}
    q_k\alpha 
&\equiv (-1)^k\delta_k \pmod 1,
\\\lambda_k &= \frac{1}{\alpha_{k+1}+ \cev{\alpha}_{k}},\\
\label{deltaratio} 
\frac{\delta_{k+2}}{\delta_k}&<\frac{1}{2},\\
\label{deltarec}
\delta_{k+1}&=\delta_{k-1}-a_{k+1}\delta_k\le \delta_{k-1}-\delta_k,\\
\label{deltasum} 
\delta_k&=\sum\limits_{t=1}^{\infty}a_{k+2t}\delta_{k+2t-1}.
\end{align}
For each $j\ge 1$ denote
\begin{equation}
\label{lambda_kj_def}
\lambda_{k,j}=q_k\delta_{k+j}=\frac{q_k}{q_{k+j}}\lambda_{k+j}.
\end{equation}
In this article, we will frequently estimate $\lambda_{k,j}$ when the partial quotients $(a_k,\ldots,a_{k+j}$) are fixed, allowing us to get reasonably sharp estimates on the ratio $q_k/q_{k+j}$ by using \textit{continuants}. For any finite (possibly empty) sequence of positive integers $C=(c_1,\ldots,c_t)$, we define the continuant $\langle C\rangle$ as follows:
\begin{enumerate}
\item{If $C$ is an empty sequence, $\langle C\rangle=1$.}
\item{If $C=(c_1)$, then $\langle c_1\rangle=c_1$.}
\item{If the length of $C$ is greater than $1$, then 
\begin{equation*}
\langle c_1,c_2,\ldots,c_t\rangle=c_t\langle c_1,c_2,\ldots,c_{t-1}\rangle+\langle c_1,c_2,\ldots,c_{t-2}\rangle.
\end{equation*}
}
\end{enumerate}
One can easily deduce from the definitions that
\begin{align*}
[0;a_1,a_2,\ldots,a_t]&=\frac{\langle a_2,\ldots,a_t\rangle}{\langle a_1,a_2,\ldots,a_t\rangle},\\
\langle a_1,\ldots,a_k,a_{k+1},\ldots,a_s\rangle&=\langle a_1,\ldots,a_{k}\rangle\langle a_{k+1},\ldots,a_{s}\rangle+
\langle a_1,\ldots,a_{k-1}\rangle\langle a_{k+2},\ldots,a_{s}\rangle,\\  
\frac{q_{k+j}}{q_k}&=\langle a_{k+1},a_{k+2},\ldots,a_{k+j}\rangle+\cev{\alpha}_k\langle a_{k+2},a_{k+3},\ldots,a_{k+j}\rangle,\quad j \ge 1.
\end{align*}
Fixing an irrational $\alpha = [a_0;a_1,a_2,\ldots]$, the Ostrowski expansion of a non-negative integer $N$ is the unique representation

\begin{equation*}
N = \sum_{\ell = 0}^k b_{\ell}q_{\ell}, \quad \text{ where }
b_{k} \neq 0,\;\; 0 \le b_0 < a_1, \;\;  0 \le b_{\ell} \le a_{\ell+1} \text { for } \ell \ge 1,
\end{equation*}
with the additional rule that
$b_{\ell-1} = 0$ whenever $b_{\ell} = a_{\ell+1}$.
If $\alpha = [a_0;a_1,\ldots,a_{k}] = p_k/q_k$ is a rational number, then we define for
$N < q_k$ the Ostrowski expansion in the same way as in the irrational setting.\\

We conclude this subsection with the following estimate of Ostrowski \cite{ostrowski}:
Let $\ell > 10$ an integer, $\alpha = [a_0;a_1,a_2,\ldots]$ irrational or $\alpha = \frac{p_k}{q_k}$
with $\ell < q_k$. Then 

\begin{equation}
\label{Ostr_est}
\left|\sum\limits_{n=1}^{\ell}1/2-\{n\alpha\}\right|\le\frac{3}{2}\log{\ell}\max_{i \le k}a_i(\alpha).
\end{equation}

\subsection*{Perturbed Sudler products and convergence of the Sudler product along subsequences}
The proof of Theorem \ref{thm_main} relies on a decomposition approach that was implicitly used in \cite{grepstad_neum} to prove $\liminf\limits_{N \to \infty}P_N(\phi) > 0$, and more explicitly in even more recent articles such as \cite{quantum_invariants,zag_conj,other_aist_borda,tech_zaf,grepstadII,hauke_density,hauke_extreme,hauke_badly}. The Sudler product $P_N(\alpha)$ is
decomposed into a finite product with factors of the form

 \begin{equation*}
P_{q_n}(\alpha,\varepsilon) := \prod_{r=1}^{q_n} 2 \left\lvert \sin\Big(\pi\Big(r\alpha + (-1)^{n}\frac{\varepsilon}{q_n}\Big)\Big)\right\rvert,\quad \varepsilon \in \mathbb{R}
\end{equation*}
We state the fact in the formulation of \cite[Proposition 4]{hauke_badly}:
\begin{prop}\label{prop_shifted}
Let $\alpha$ be a fixed irrational and let $N = \sum_{i=0}^{n} b_{i}q_i(\alpha)$ be the Ostrowski expansion of an arbitrary integer $q_n \le N < q_{n+1}$.
For $0 \le i \le n$ and $t \in \mathbb{N}$, we define
\begin{equation*}
\varepsilon_{i,t}(N) = \varepsilon_{i,t}(\alpha,N) := q_i\left(t\delta_i + \sum_{j=1}^{n-i} (-1)^jb_{i+j} \delta_{i+j}\right).
\end{equation*}
Then we have

\begin{equation*}
    P_N(\alpha) = \prod_{i=0}^{n}\prod_{t= 0}^{b_i-1} P_{q_i}\bigl(\alpha,\varepsilon_{i,t}(N)\bigr).
\end{equation*}
\end{prop}
Since the set $\{n\alpha\}_{n=1}^{q_k}$ is (for badly approximable $\alpha$) very regularly distributed in the unit interval, the Sudler product at stage $q_k$ and its perturbed variant are much easier to control than the Sudler product at an arbitrary stage $N$. The first result in this direction was found by Mestel and Verschueren \cite{mestel} who proved that along the Fibonacci sequence $(F_n)_{n \in \N}$ (which are the convergents' denominators of the Golden Ratio), the limit $\lim_{N \to \infty} P_{F_n}(\phi)$ exists and equals $C \approx 2.407$. This convergence phenomenon was extended in \cite{tech_zaf}, showing that for $\beta(b) = [0;\overline{b}]$, we have
\begin{equation}\label{limit_fct}\lim_{N \to \infty} P_{q_n(\beta)}(\beta,\varepsilon) = G_{\beta}(\varepsilon),\end{equation}
where $G_{\beta}$ is a continuous function that can be described in a closed expression. Furthermore, on intervals that are bounded away from the zeros of $G_{\beta}$, the function is even $C^{\infty}$ and strictly log-concave. The convergence in \eqref{limit_fct} is locally uniform and even a convergence rate was established, see \cite{quantum_invariants,hauke_extreme}.

The statement \eqref{limit_fct} was generalized in \cite{grepstadII} to arbitrary quadratic irrationals: Given a period length $p$, the convergence of $P_{q_{np + m}}(\alpha,\varepsilon) \xrightarrow[n \to \infty]{} G_{m,\alpha}(\varepsilon),\ m = 0,\ldots,p-1$ was proven, with $G_{m,\alpha}$ having the same properties as $G_{\beta}$ mentioned above. Due to the lack of periodicity, a similar convergence result does not make sense for general badly approximable numbers; however, the following generalization to badly approximable numbers was established in \cite{hauke_badly}:

\begin{prop}\cite[Proposition 7]{hauke_badly}.\label{limit_H}
For $h \in \mathbb{N}$, we define 
\begin{equation*}
H_k(\alpha,\varepsilon) :=
2\pi \lvert \varepsilon + \lambda_k \rvert \prod_{n=1}^{\lfloor q_k/2\rfloor } h_{n,k}(\varepsilon),
\end{equation*}
where
\begin{equation}
\label{hkdef}
h_{n,k}(\varepsilon) = h_{n,k}(\alpha,\varepsilon) := \Bigg\lvert\Bigg(1 - \lambda_k\frac{\left\{n\cev{\alpha}_k\right\} - \frac{1}{2}}{n}\Bigg)^2 - \frac{\left(\varepsilon + \frac{\lambda_k}{2}\right)^2}{n^2}\Bigg\rvert.
\end{equation}
For any badly approximable $\alpha$ and any compact interval $I$, we have 
\[P_{q_k}(\alpha,\varepsilon) = H_k(\alpha,\varepsilon)\left(1 + \mathcal{O}\left(q_k^{-2/3}\log^{2/3}q_k\right)\right) + \mathcal{O}(q_k^{-2}),\quad \varepsilon \in I,\]
with the implied constant only depending on $K(\alpha) = \max_{i \in \mathbb{N}}a_i(\alpha)$ and $I$.\footnote{Although stated in \cite{hauke_badly} with dependence on $\alpha$, the method of proof actually reveals that the implied constant only depends on $K(\alpha)$.} In particular, we have

\[\lim_{k \to \infty} \inf_{\alpha \in E_3}\lvert P_{q_k}(\alpha,\varepsilon) - H_k(\alpha,\varepsilon) \rvert = 0,\]
with the convergence being locally uniform on $\mathbb{R}$.
\end{prop}

\section{The main ideas -- proof of the theorems assuming technical lemmas}\label{main_ideas}

We observe the following two important facts:
\begin{itemize}
    \item The functions $H_k(\alpha,\varepsilon)$ only depend (up to small, controllable factors) on partial quotients $a_i(\alpha)$ where $\lvert i-k\rvert$ is small.
    \item In this paper, we are only interested in finding \textit{lower} bounds for $H_k(\alpha,\varepsilon)$.
\end{itemize}

Therefore, we do the following: We classify the pairs $(\alpha,k)$ by the nine partial quotients 
$$
\cc(\alpha,k):=(a_{k-3},a_{k-2},\ldots, a_{k+4},a_{k+5})
$$
and construct $3^9=19683$ functions $F_{\cc}$ such that
$\liminf_{\ell \to \infty}H_{k_{\ell}}(\alpha,\varepsilon) \ge F_{\cc}(\varepsilon)$ on a certain interval $(\varepsilon_{\cc}^{\min},\varepsilon_{\cc}^{\max})$ where $k_{\ell}(\alpha)$ is the subsequence of $k$ satisfying $(a_{k-3}(\alpha),a_{k-2}(\alpha),\ldots,a_{k+5}(\alpha)) = \cc$ (see Lemma \ref{new_idea_lem}). In that way, we are only left with finitely many functions instead of the infinite sequence $(H_k)_{k \in \N}$.\\
The following proposition, in which we also introduce some notation, follows from basic properties of continued fractions. We recall that the quantities $\alpha_k$, $\cev{\alpha}_k$ and $\lambda_k,\lambda_{k,j}$ are defined in (\ref{alpha_delta_def}) and (\ref{lambda_kj_def}).
\begin{defprop}
\label{prop_lambda_def}
For each $\cc = (c_1,\ldots,c_9)\in \{1,2,3\}^9$ and $k \ge 3$, we define 
$$
\mathfrak{A}(\cc,k):=\left\{\alpha\in E_3 | (a_{k-3}(\alpha),a_{k-2}(\alpha),\ldots, a_{k+4}(\alpha),a_{k+5}(\alpha))=\cc \right\}.
$$
Furthermore, we denote
\begin{align*}
\vec{\cc}_{\min} &:= \lim_{k \to \infty}\inf\limits_{\alpha\in\mathfrak{A}(\cc,k)}\alpha_{k+1}(\alpha) =[c_5;c_6,c_7,c_8,c_9,\overline{3,1}],\\
\vec{\cc}_{\max}&:=\lim_{k \to \infty}\sup\limits_{\alpha\in\mathfrak{A}(\cc,k)}\alpha_{k+1}(\alpha) =[c_5;c_6,c_7,c_8,c_9,\overline{1,3}] ,\\
\cev{\cc}_{\min}&:=\lim_{k \to \infty}\inf\limits_{\alpha\in\mathfrak{A}(\cc,k)}\cev{\alpha}_{k}(\alpha)  = [0;c_4,c_3,c_2,c_1,\overline{3,1}],\\
\cev{\cc}_{\max}&:=\lim_{k \to \infty}\sup\limits_{\alpha\in\mathfrak{A}(\cc,k)}\cev{\alpha}_{k}(\alpha) =[0;c_4,c_3,c_2,c_1,\overline{1,3}] ,\\
\lambda_{\cc}^{\min}&:=\lim_{k \to \infty}\inf\limits_{\alpha\in\mathfrak{A}(\cc,k)}\lambda_{k}(\alpha)  = \frac{1}{\vec{\cc}_{\max}+\cev{\cc}_{\max}},\\
\lambda_{\cc}^{\max}&:=\lim_{k \to \infty}\sup\limits_{\alpha\in\mathfrak{A}(\cc,k)}\lambda_{k}(\alpha)=\frac{1}{\vec{\cc}_{\min}+\cev{\cc}_{\min}}.
\end{align*}
For any $j \in \{1,2,3,4\}$ we denote for odd $j$
\begin{align*}
\lambda_{\cc, j}^{\min}&:=\lim_{k \to \infty}\inf\limits_{\alpha\in\mathfrak{A}(\cc,k)}\lambda_{k,j}(\alpha)= \frac{1}{\langle c_5,\ldots,c_{5+j-1}\rangle+\cev{\cc}_{\max}\langle c_6,\ldots,c_{5+j-1}\rangle}\frac{1}{[c_{5+j};c_{5+j+1},\ldots,c_9,\overline{3,1}]+[0;c_{5+j-1},\ldots,c_1,\overline{3,1}]},\\
\lambda_{\cc, j}^{\max}&:=\lim_{k \to \infty}\sup\limits_{\alpha\in\mathfrak{A}(\cc,k)}\lambda_{k,j}(\alpha)= \frac{1}{\langle c_5,\ldots,c_{5+j-1}\rangle+\cev{\cc}_{\min}\langle c_6,\ldots,c_{5+j-1}\rangle}\frac{1}{[c_{5+j};c_{5+j+1},\ldots,c_9,\overline{1,3}]+[0;c_{5+j-1},\ldots,c_1,\overline{1,3}]}.
\end{align*}
If $j$ is even, we replace the period $\overline{3,1}$ by $\overline{1,3}$ and vice-versa.
Similarly, we define
\begin{align*}
\lambda_{\cc, 5}^{\min}&:=\lim_{k \to \infty}\inf\limits_{\alpha\in\mathfrak{A}(\cc,k)}\lambda_{k,j}(\alpha)= \frac{1}{\langle c_5,\ldots,c_{9}\rangle+\cev{\cc}_{\max}\langle c_6,\ldots,c_{9}\rangle}\frac{1}{[3;\overline{1,3}]+[0;c_{9},\ldots,c_1,\overline{3,1}]},\\
\lambda_{\cc, 5}^{\max}&:=\lim_{k \to \infty}\sup\limits_{\alpha\in\mathfrak{A}(\cc,k)}\lambda_{k,j}(\alpha)= \frac{1}{\langle c_5,\ldots,c_{9}\rangle+\cev{\cc}_{\min}\langle c_6,\ldots,c_{9}\rangle}\frac{1}{[1;\overline{3,1}]+[0;c_{9},\ldots,c_1,\overline{1,3}]}.
\end{align*}
\end{defprop}

In view of Proposition \ref{prop_shifted}, every perturbation variable $\varepsilon$ that needs to be considered arises
from some $\varepsilon_{i,c_i}(N)$. The following statement provides both upper and lower bounds for such perturbations:

\begin{prop}\label{max_perturb}
Let $\alpha \in \tilde{E}_3$ fixed. For all $i\in\mathbb{N}$ and every $N \in \N$ we have 
\begin{equation}
\label{eps_min_max}
 -\lambda_i+\lambda_{i,1}\le \varepsilon_{i,c_i}(N)\le (a_{i+1}-1)\lambda_i+\lambda_{i,1}.
\end{equation}
In particular, if $i$ is sufficiently large and $(a_{i-3}(\alpha),a_{i-2}(\alpha),\ldots, a_{i+4}(\alpha),a_{i+5}(\alpha))=\cc$, we have for every $N \in \N$
\begin{equation}
\label{eps_min_max_new}
-1 < \varepsilon_{\cc}^{\min} := -\lambda_{\cc}^{\max}+\lambda_{\cc,1}^{\min} - 10^{-100}\le \varepsilon_{i,c_i}(N) \le (c_5-1)\lambda_{\cc}^{\max}+\lambda_{\cc,1}^{\max} + 10^{-100} =: \varepsilon_{\cc}^{\max} < 1.
\end{equation}
\end{prop}
\begin{proof}
Using \eqref{deltasum} and the fact that $b_{i+1}\le a_{i+2}-1$, we obtain that
$$
\varepsilon_{i,c_i}(N)\ge\varepsilon_{i,0}(N)\ge -q_i\left((a_{i+2}-1)\delta_{i+1}-a_{i+4}\delta_{i+3}-a_{i+6}\delta_{i+5}-\ldots\right)\ge -q_i(\delta_i-\delta_{i+1})=-\lambda_i+\lambda_{i,1}.
$$
On the other hand, using also the fact that $b_i\le a_{i+1}$, we have
$$
\varepsilon_{i,c_i}(N)\le\varepsilon_{i,b_i-1}(N)=q_i\left((b_i-1)\delta_i+a_{i+3}\delta_{i+2}+a_{i+5}\delta_{i+4}+\ldots\right)\le q_i((a_{i+1}-1)\delta_i+q_i\delta_{i+1})=(a_{i+1}-1)\lambda_i+\lambda_{i,1},
$$
which shows \eqref{eps_min_max}. The statement in \eqref{eps_min_max_new} now follows immediately from \eqref{eps_min_max} and Definition/Proposition \ref{prop_lambda_def}.
\end{proof}

Having this established, we present Lemma \ref{new_idea_lem}, which is, from a technical perspective, the main ingredient of this article. It can be seen as an analog of the convergence along subsequences established in \cite{tech_zaf,grepstadII}.

\begin{lem}\label{new_idea_lem}
    Let $n_0 \ge 10, T \ge n_0+2$ and $m \in \N$ be fixed parameters. For any $\cc \in \{1,2,3\}^9$, let
\begin{equation}
\label{Hk_step_fin_}
 F_{\cc}(\varepsilon)= F_{\cc}(\varepsilon,n_0,T,m) :=
2\pi(\varepsilon+\lambda_{\cc}^{\min})f_{\cc, \infty}(\varepsilon) \prod_{n=1}^{n_0 } f_{\cc,n}(\varepsilon), \quad \varepsilon\in(\varepsilon_{\cc}^{\min},\varepsilon_{\cc}^{\max}),
\end{equation}
where the positive and log-concave functions 
$f_{\cc,n}(\varepsilon)$ and $f_{\cc,\infty}(\varepsilon)=f_{\cc,\infty}(\varepsilon,n_0,T,m)$ will be defined later.
Then we have the following:
\begin{itemize}
    \item $F_{\cc}(\varepsilon)$ is pseudo-concave on $(\varepsilon_{\cc}^{\min},\varepsilon_{\cc}^{\max})$.
    \item For any $\alpha \in E_3$, let $k_{\ell}(\cc,\alpha):= \{k \in \N: (a_{k-3}(\alpha),a_{k-2}(\alpha),\ldots, a_{k+4}(\alpha),a_{k+5}(\alpha))=\cc\}$.
     Then we have 
\[
\liminf_{\ell \to \infty}H_{k_{\ell}}(\alpha,\varepsilon) \ge F_{\cc}(\varepsilon)
\]
uniformly in $(\varepsilon_{\cc}^{\min},\varepsilon_{\cc}^{\max})$
and $E_3$.
\end{itemize}
\end{lem}

The construction of the family $\mathcal{F} := \{F_{\cc}: \cc \in \{1,2,3\}^9\}$ and the proof of Lemma \ref{new_idea_lem} will be given in Section \ref{Hk_section}.
It turns out that after an extensive case distinction and computational assistance that will be proven in Section \ref{case_study}, the following holds true:

\begin{lem}\label{main_techn_lem_variant}
Let $F_{\cc} = F_{\cc}(m,n_0,T)$ be defined as in \eqref{Hk_step_fin_}. With parameters $m = 40,n_0 = 20$ and $T = 10000$, the following inequalities hold:
\begin{equation}
\label{mainineq_negpert_}
 F_{\cc}(\varepsilon_{\cc}^{\min})>0.35,
\end{equation}

\begin{equation}
\label{mainineq_zeropert_}
 F_{\cc}(0)>1.0001,
\end{equation}
\begin{equation}
\label{mainineq_no_large_positive_}
F_{\cc}\bigl(\varepsilon_{\cc}^{\max}\bigr)>1.0001.
\end{equation}

Furthermore, for any $\alpha \in \tilde{E}_3$ we have the following: For $N \in \N$, let $i_k=i_k(N)$ denote the sequence of indices such that the Ostrowski digit $b_{i_k}=b_{i_k}(N)$ is non-zero.
Then for all $N$ large enough and writing $\cc_k = \cc(\alpha,k)$, every $F_{\cc_k} \in \mathcal{F}$ satisfies

\begin{equation}
\begin{split}
\label{mainineq_no_large_negative_}
\max\Bigl(&F_{\cc_{i_{k}}}(\varepsilon_{{i_{k}},0}(N)), 
F_{\cc_{i_{k}}}(\varepsilon_{{i_{k}},0}(N))\prod_{t=1}^{b_{i_k+1}-1} F_{\cc_{i_{k+1}}}(\varepsilon_{i_{k+1},t}(N)), F_{\cc_{i_{k}}}(\varepsilon_{{i_{k}},0}(N))F_{\cc_{i_{k+1}}}(\varepsilon_{{i_{k+1}},0}(N)),\\&
F_{\cc_{i_{k}}}(\varepsilon_{{i_{k}},0}(N))F_{\cc_{i_{k+1}}}(\varepsilon_{{i_{k+1}},0}(N))F_{\cc_{i_{k+2}}}(\varepsilon_{{i_{k+2}},0}(N))\Bigr)>1.0001.
\end{split}
\end{equation}
\end{lem}

 \begin{proof}[Proof of Theorem \ref{thm_main} assuming Lemmas \ref{new_idea_lem} and \ref{main_techn_lem_variant}]

Let $k_0$ be a large integer such that for every $k \ge k_0$ and any $N \in \N$, we have the following:

\begin{itemize}
    \item 
    For any $0 \le t \le b_{k-1}$, we have by  \eqref{eps_min_max_new}
    \begin{equation}\label{perturb_range}\varepsilon_{k,t}(N) \in (\varepsilon_{\cc_k}^{\min},\varepsilon_{\cc_k}^{\max}).\end{equation}
\item By Proposition \ref{limit_H} and \eqref{perturb_range},\begin{equation}\label{PH_close}P_{q_{k}}(\alpha,\varepsilon_{k,c_k}(N)) \ge 
H_{k}(\alpha,\varepsilon_{k,c_k}(N)) - 10^{-10}.\end{equation}
\item By  Lemma \ref{new_idea_lem} and \eqref{perturb_range}, \begin{equation}\label{HF_close}H_{k}(\alpha,\varepsilon_{k,c_k}(N)) \ge F_{\cc_k}(\varepsilon_{k,c_k}(N)) - 10^{-10}.\end{equation}
\end{itemize}

Now let $N$ be a fixed integer. Denote by $m=m(N)$ the number of non-zero Ostrowski digits of $N$.
By Proposition \ref{prop_shifted}, we can write
\[
    P_N(\alpha) = \prod_{k=1}^{m}\prod_{t= 0}^{b_{i_k}-1} P_{q_{i_k}}(\alpha,\varepsilon_{i,t}(N)).
\] 
Writing $m_0$ for the smallest index such that $i_{m_0}>k_0$, we use \eqref{perturb_range}, \eqref{PH_close} and \eqref{HF_close} to obtain
\begin{equation}\label{liminf_}
\prod_{k=1}^{m}\prod_{t= 0}^{b_{i_k}-1} P_{q_{i_k}}(\alpha,\varepsilon_{i,t}(N))
\ge \left(\prod_{k=1}^{m_0-1}\prod_{t=0}^{b_{i_k}-1} P_{q_{i_k}}(\varepsilon_{i_k,t}(N))\right)\cdot
\left(\prod_{k=m_0}^{m}\prod_{t= 0}^{b_{i_k}-1} \left(F_{\cc_{i_k}}(\varepsilon_{i_k,t}(N))- 10^{-9}\right)\right).
\end{equation}
Since by \cite[Lemma 3] {quantum_invariants} $P_{q_k}(\alpha, \varepsilon_{i,c_i}(N))$ is bounded away from $0$, it suffices to consider the product of the remaining factors of (\ref{liminf_}).
Note that if $\varepsilon_{i_k,t}(N) \ge 0$, we have $F_{\cc_{i_k}}(\varepsilon_{i_k,t}(N))- 10^{-9} \ge 1.00001$. This follows from  (\ref{mainineq_zeropert_}), (\ref{mainineq_no_large_positive_}) and the pseudo-concavity of $F_{\cc}$.
If $t\ge 1$, then
$$
\varepsilon_{\cc}^{\min}>\varepsilon_{i_k,t}(N)=
q_i\left(t\delta_i + \sum_{j=1}^{n-i} (-1)^jb_{i+j} \delta_{i+j}\right)\ge
q_i\left(\delta_i-\sum_{j=0}^{\infty} a_{i+j+2} \delta_{i+j+1}\right)\ge 0,
$$
thus for $t\ge 1$,
\begin{equation}\label{use_logconcave}F_{\cc_{i_k}}(\varepsilon_{i_k,t}(N))- 10^{-9} \ge 1.00001.\end{equation}
This implies that

\begin{equation*}
\prod_{k=m_0}^{m}\prod_{t=0}^{b_{i_k}-1} F_{\cc_{i_k}}(\varepsilon_{i_k,t}(N))\ge
F_{\cc_{i_{m_0}}}(\varepsilon_{i_{m_0},0}(N))\prod_{k=m_0+1}^{m}\prod_{t=0}^{b_{i_k}-1} F_{\cc_{i_k}}(\varepsilon_{i_k,t}(N)).
\end{equation*}
Thus in order to complete the proof, it remains to show that
\begin{equation}
\label{product_power_m}
F_{\cc_{i_{m_0}}}(\varepsilon_{i_{m_0},0}(N))\prod_{k=m_0+1}^{m}\prod_{t=0}^{b_{i_k}-1} F_{\cc_{i_k}}(\varepsilon_{i_k,t}(N))\ge 0.01\cdot 1.00001^{(m-m_0)/3-1}.
\end{equation}
We do this by using induction on $p=m-m_0$.
If $m-m_0 \le 3$, the statement follows from the estimates (\ref{mainineq_negpert_}), (\ref{mainineq_no_large_positive_}) and (\ref{use_logconcave}).
Suppose that (\ref{product_power_m}) holds for all $m-m_0\le p$ where $p \ge 3$. We prove this inequality for $m-m_0=p+1 \ge 4$. 
We distinguish cases depending on which term in \eqref{mainineq_no_large_negative_} is maximal for $i_k = i_{m_0}$.

If the first term is maximal, then
$F_{\cc_{i_{m_0}}}(\varepsilon_{{i_{m_0}},0}(N))>1.0001.$
In this case, we apply \eqref{use_logconcave} for $\varepsilon_{{i_{m_0+1}},t}(N)$ for any $1 \le t \le b_{i_{m_0+1}}-1$ and use the induction hypothesis for $m-m_0 -1$ on the remaining factors to show \eqref{product_power_m}.

If the maximum is attained at the second term of (\ref{mainineq_no_large_negative_}), we have 

\[
F_{\cc_{i_{m_0}}}(\varepsilon_{i_{m_0},0}(N))\prod_{k=m_0+1}^{m}\prod_{t=0}^{b_{i_k}-1} F_{\cc_{i_k}}(\varepsilon_{i_k,t}(N)) > F_{\cc_{i_{m_0+1}}}(\varepsilon_{i_{m_0+1},0}(N))\prod_{k=m_0+2}^{m}\prod_{t=0}^{b_{i_k}-1} F_{\cc_{i_k}}(\varepsilon_{i_k,t}(N))
\]
and we apply again the induction hypothesis for $m - m_0 -1$.

If the maximum is attained at the third or fourth term, we remove by \eqref{use_logconcave} again contributions of $t \geq 1$ and apply the induction hypothesis for $m-m_0-2$ respectively $m-m_0-3$. 
 This proves \eqref{product_power_m} which finishes the proof.
\end{proof}

\begin{rmk*}
    We note that the proof above actually shows \eqref{uniform_bound}:
    One can verify that $k_0$ in the above proof can be chosen uniformly in $\alpha \in E_3$ and since for every $n \in \N$, $\inf_{\alpha \in E_3} \lVert n\alpha \rVert >0$, all factors $P_{q_k}$ with $k < k_0$ are also uniformly bounded away from $0$.
\end{rmk*}

\section{Construction of $\mathcal{F}$}
\label{Hk_section}

In this section, we prove Lemma \ref{new_idea_lem}, that is, we construct pseudo-concave functions $F_{\cc(\alpha,k)}(\varepsilon)$ that serve as lower bounds of $H_k(\alpha,\varepsilon)$. We start with an elementary proposition that contains estimates that will be useful throughout the subsequent calculations.

\begin{prop}\label{useful_estimates}
    For any $\cc \in \{1,2,3\}^9$, we have the following estimates:

\begin{align}
\label{325}
\cev{\cc}_{\max}-\cev{\cc}_{\min}&<\frac{3}{25}.\\
\label{useful_2}\varepsilon_{\cc}^{\min} &\ge -\lambda_{\cc}^{\min}.
\end{align}
\end{prop}
\begin{proof}
    We estimate $\cev{\cc}_{\max}-\cev{\cc}_{\min}$ using the well-known formula
\begin{equation}
\label{alpha_beta_diff}
[a_0;a_1,\ldots,a_n,\beta_{n+1}]-[a_0;a_1,\ldots,a_n,\alpha_{n+1}]=(-1)^{n+1}\frac{\beta_{n+1}-\alpha_{n+1}}{q^2_n(\alpha_{n+1}+\cev{\alpha_{n}})(\beta_{n+1}+\cev{\alpha_{n}})}.
\end{equation}
  As the first four partial quotients of $\cev{\cc}_{\max}$, $\cev{\cc}_{\min}$ coincide, 
  we apply (\ref{alpha_beta_diff}) for $n = 4$. Note that $q_4\ge F_5=5$ in (\ref{alpha_beta_diff}), where $F_n$ is $n$-th Fibonacci number. Using trivial bounds on the right-hand side of (\ref{alpha_beta_diff}), we obtain \eqref{325}.\\
  Note that \eqref{useful_2}
  is equivalent to showing $\lambda_{\cc}^{\max} - \lambda_{\cc,1}^{\min} \ge \lambda_{\cc}^{\min} + \dk$. Using \eqref{325}, we observe that
  \begin{equation*}
\lambda_{\cc}^{\max}-\lambda_{\cc}^{\min}=\frac{(\vec{\cc}_{\max}-\vec{\cc}_{\min})+(\cev{\cc}_{\max}-\cev{\cc}_{\min})}{(\vec{\cc}_{\min}+\cev{\cc}_{\min})(\vec{\cc}_{\max}+\cev{\cc}_{\max})}\le
\frac{6}{25(c_5+0.5)^2}.
\end{equation*}
Since $q_{k}/q_{k+1}<1/{(c_5+1)}$, we see that
$
\lambda_{\cc,1}^{\min}\ge\frac{1}{4.6(c_5+1)}.
$
and by $c_5 \ge 1$, the statement follows.
\end{proof}

We recall that
$$
H_k(\alpha,\varepsilon) =
2\pi \lvert \varepsilon + \lambda_k \rvert \prod_{n=1}^{\lfloor q_k/2\rfloor } h_{n,k}(\alpha,\varepsilon),
$$
where $h_{n,k}$ is as in (\ref{hkdef}). 
By Definition/Proposition \ref{prop_lambda_def}, for any $\alpha \in \mathfrak{A}(\cc,k)$ we have
\footnote{Here and in what follows, the statements should be interpreted as follows:
There exists a function $\delta(k)$ such that $\lim_{k \to \infty}\delta(k) = 0$
and with $k_{\ell}(\cc,\alpha):= \{k \in \N: (a_{k-3}(\alpha),a_{k-2}(\alpha),\ldots, a_{k+4}(\alpha),a_{k+5}(\alpha))=\cc\}$, we have e.g. 
\[
\alpha_{k_{\ell}(\cc)+1}(\alpha)+ \delta(k_{\ell}(\cc)) \ge \vec{\cc}_{\min}.
\] Since $\{1,2,3\}^9$ is a finite set, the convergence is uniform in $\cc$ and by definition of $\vec{\cc}_{\min}$, also uniform in $E_3$.}
\begin{equation}
\label{prop8est}
\vec{\cc}_{\min}-o(1) \le \alpha_{k+1}(\alpha)\le \vec{\cc}_{\max}+o(1),\quad
\cev{\cc}_{\min}-o(1)\le \cev{\alpha}_{k}(\alpha) \le \cev{\cc}_{\max}+o(1), \quad
\lambda_{\cc}^{\min}-o(1)\le \lambda_{k}(\alpha)\le \lambda_{\cc}^{\max}+o(1).
\end{equation}

By \eqref{useful_2}, for any $\varepsilon \in (\varepsilon_{\cc}^{\min},\varepsilon_{\cc}^{\max})$ we have
$\varepsilon+\lambda_{\cc}^{\min} >0$ and hence by \eqref{prop8est},
\begin{equation}
\label{Hk_step1}
H_k(\varepsilon)\ge
2\pi(\varepsilon+\lambda_{\cc}^{\min}-o(1)) \prod_{n=1}^{\lfloor q_k/2\rfloor } h_{n,k}(\varepsilon), \quad  \varepsilon \in (\varepsilon_{\cc}^{\min},\varepsilon_{\cc}^{\max}).
\end{equation}
The factors $h_{n,k}(\varepsilon)$ will be estimated as follows. For a fixed small parameter $n_0 \in \N$ and $1\le n\le n_0$ we will construct functions $f_{\cc, n}(\varepsilon)$ such that
for any $\alpha \in \mathfrak{A}(\cc,k)$
\begin{equation}
\label{fn_cond}
h_{n,k}(\alpha,\varepsilon) + o(1) \ge f_{\cc, n}(\varepsilon),\quad \varepsilon\in(\varepsilon_{\cc}^{\min},\varepsilon_{\cc}^{\max}).
\end{equation}
Further we construct a function $f_{\cc,\infty}(\varepsilon)$ such that
\begin{equation}
\label{fn_cond_inf}
\prod_{n=n_0+1}^{\lfloor q_k/2\rfloor }h_{n,k}(\alpha,\varepsilon)\ge f_{\cc, \infty}(\varepsilon) - o(1),\quad \varepsilon\in(\varepsilon_{\cc}^{\min},\varepsilon_{\cc}^{\max}) .
\end{equation}
Combining (\ref{Hk_step1}), (\ref{fn_cond}) and (\ref{fn_cond_inf}) we obtain
\begin{equation*}
H_k(\alpha,\varepsilon) + o(1)\ge
2\pi(\varepsilon+\lambda_{\cc}^{\min}) f_{\cc, \infty}(\varepsilon) \prod_{n=1}^{n_0 } f_{\cc,n}(\varepsilon)=:F_{\cc}(\varepsilon).
\end{equation*}

\subsubsection*{Construction of $f_{\cc,n}$.}
\begin{lem}\label{fcn_lem}
\label{f_n_def}
For any $\cc \in \{1,2,3\}^9$ and $1 \le n \le n_0$, let
\begin{equation*}
g_{\cc,n}:=
\begin{cases}
\Bigl(1-\lambda_{\cc}^{\max}\frac{\{n \cev{\cc}_{\max}\}-0.5}{n}\Bigr)^2\quad \text{if}\ \{n\cev{\cc}_{\max}\}\ge 0.5\  \text{and}\  \lfloor n \cev{\cc}_{\max}\rfloor = \lfloor n \cev{\cc}_{\min}\rfloor , \\
\Bigl(1-\lambda_{\cc}^{\min}\frac{\{n\cev{\cc}_{\max}\}-0.5}{n}\Bigr)^2 \quad \text{if}\ \{n\cev{\cc}_{\max}\}<0.5\  \text{and}\  \lfloor n \cev{\cc}_{\max}\rfloor = \lfloor n \cev{\cc}_{\min}\rfloor,\\
\Bigl(1-\frac{\lambda_{\cc}^{\max}}{2n}\Bigr)^2 \quad \text{if} \ \
\lfloor n\cev{\cc}_{\max}\rfloor\ne\lfloor n \cev{\cc}_{\min}\rfloor,
\end{cases}
\end{equation*}
and
\begin{equation*}
e_{\cc,n}(\varepsilon):=\max\biggl(\frac{(\varepsilon+\lambda_{\cc}^{\max}/2)^{2}}{n^2}, \frac{(\varepsilon+\lambda_{\cc}^{\min}/2)^{2}}{n^2}\biggr).
\end{equation*}

Then $f_{\cc,n}(\varepsilon):=g_{\cc,n}-e_{\cc,n}(\varepsilon)$ is concave. Furthermore, we have
\begin{equation}\label{fn_cond_}h_{n,k}(\alpha,\varepsilon)\ge f_{\cc(\alpha,k), n}(\varepsilon) - o(1),\quad \varepsilon\in (\varepsilon_{\cc}^{\min},\varepsilon_{\cc}^{\max}).\end{equation}
\end{lem}
\begin{proof}
Since $g_{\cc,n}$ does not depend on $\varepsilon$, concavity follows from the fact that
$e_{\cc,n}(\varepsilon)$ is, as the maximum of two convex functions, convex.
Thus we are left to prove that (\ref{fn_cond_}) is satisfied.
Trivally,
\[ h_{n,k}(\alpha,\varepsilon) := \Bigg\lvert\Bigg(1 - \lambda_k\frac{\left\{n\cev{\alpha}_k\right\} - \frac{1}{2}}{n}\Bigg)^2 - \frac{\left(\varepsilon + \frac{\lambda_k}{2}\right)^2}{n^2}\Bigg\rvert
\ge \Bigg(1 - \lambda_k\frac{\left\{n\cev{\alpha}_k\right\} - \frac{1}{2}}{n}\Bigg)^2 - \frac{\left(\varepsilon + \frac{\lambda_k}{2}\right)^2}{n^2}.\]
Since 
\begin{equation}\label{close_to_lambda}
\lambda_{\cc}^{\min}-o(1) \le \lambda_k \le  \lambda_{\cc}^{\max} + o(1),
\end{equation}
we have $\frac{\left(\varepsilon+\frac{\lambda_k}{2}\right)^2}{n^2}\le  e_{\cc,n}(\varepsilon) + o(1)$, thus we are left to prove that
\begin{equation}
    \label{close_to_g}
\Bigg(1 - \lambda_k\frac{\left\{n\cev{\alpha}_k\right\} - \frac{1}{2}}{n}\Bigg)^2\ge g_{\cc,n}- o(1), \quad k \to \infty.
\end{equation}
If 
$\lfloor n\cev{\cc}_{\max}\rfloor\ne\lfloor n \cev{\cc}_{\min}\rfloor$,
then \eqref{close_to_g} follows immediately from \eqref{close_to_lambda}, so we can assume $\lfloor n\cev{\cc}_{\max}\rfloor =\lfloor n \cev{\cc}_{\min}\rfloor$.
Since $\cev{\cc}_{\max}$ is irrational, there exists $\delta > 0$ such that
\[\min_{1 \le n \le n_0} \lVert n \cev{\cc}_{\max}\rVert > \delta.\]
Since $\cev{\alpha}_k \le \cev{\cc}_{\max} + o(1)$, for $k$ sufficiently large, $\cev{\alpha}_k \le \cev{\cc}_{\max} + \frac{\delta}{n_0}$,
which implies
$
\lfloor n\cev{\alpha}_k \rfloor \le \lfloor n\cev{\cc}_{\max} \rfloor, 1 \le n \le n_0
$ and by the same argumentation, we get $\lfloor n\cev{\cc}_{\min} \rfloor \le \lfloor n\cev{\alpha}_k \rfloor$.
Hence \eqref{close_to_lambda} implies
$
\{n \cev{\cc}_{\min}\}-o(1) \le \{n\cev{\alpha}_k\} \le \{n \cev{\cc}_{\max}\} + o(1)$, so \eqref{close_to_g} also follows if $\lfloor n\cev{\cc}_{\max}\rfloor=\lfloor n \cev{\cc}_{\min}\rfloor$.\\
\end{proof}

\subsubsection*{Construction of $f_{\cc,\infty}$.}

Next, we construct the function $f_{\cc,\infty}$ satisfying 
\begin{equation}
\label{prod_from_n0}
\prod_{n=n_0+1}^{\lfloor q_k/2\rfloor }h_{n,k}(\alpha,\varepsilon)\ge f_{\cc, \infty}(\varepsilon)  - o(1),\quad \varepsilon\in(\varepsilon_{\cc}^{\min},\varepsilon_{\cc}^{\max}).
\end{equation}
The function is defined by
\begin{equation*} 
f_{\cc,\infty}(\varepsilon)=1-\Biggl(2\lambda_{\cc}^{\max}\biggl(W_{\cc}(n_0, T,m)
 + \frac{9}{2}\frac{1+\log{T}}{T}\biggr)+\frac{1}{n_0}\biggl(\frac{(\lambda_{\cc}^{\max})^2}{4}+(\varepsilon+0.5\lambda_{\cc}^{\max})^2+E^2_{\cc}(n_0,\varepsilon)\biggr)\Biggr).
\end{equation*}
The value $W_{\cc}(n_0, T,m)$ and the function $E^2_{\cc}(n_0,\varepsilon)$ will be defined in this subsection.\\

Note that for fixed $k,\alpha,\varepsilon$, the sequence $h_{n,k}(\alpha,\varepsilon)$ tends to $1$ as $n$ grows. That is why it is convenient to analyze the logarithm of the left-hand side of (\ref{prod_from_n0}) which replaces the product with a sum. We use the following lemma \cite[Lemma 9]{hauke_extreme}, which follows from Taylor approximation of the logarithm.
\begin{prop}
\label{Hauke2021lem}
Let $x_n$ be a finite sequence of real numbers that satisfy $|x_n|<\frac{1}{2}$ and $|x_n|<\frac{c}{n}$ for some $c > 0$. Then
\begin{equation*}
\prod\limits_{n=N}^{M}(1-x_n)\ge 1-\Biggl(\Bigl|\sum\limits_{n=N}^M x_n\Bigr|+\frac{c^2}{N-1}\Biggr).
\end{equation*}
\end{prop}

\begin{lem}\label{taylor_lem}
Let $\alpha \in E_3$ and $T \ge n_0 +1\ge 11$. Defining
$
S_{\ell}(\alpha):=\sum\limits_{n=1}^{\ell}1/2-\{n\alpha\}, \ell \in \N,$ we have
 \begin{equation}\label{step_between}
 \prod_{n=n_0+1}^{\lfloor q_k/2\rfloor } h_{n,k}(\alpha,\varepsilon)\ge 
 1-\Biggl(2\lambda_{\cc}^{\max}\left(\left|\sum_{\ell=n_0+1}^{T}\frac{S_{\ell}(\cev{\alpha}_k)}{\ell(\ell +1)}\right| 
 + \frac{9}{2}\frac{1+\log{T}}{T}\right)+\frac{1}{n_0}\biggl(\frac{(\lambda_{\cc}^{\max})^2}{4}+(\varepsilon+0.5\lambda_{\cc}^{\max})^2+E^2_{\cc}(n_0,\varepsilon)\biggr)\Biggr) - o(1)
 \end{equation}
 where 
 \[
 E_{\cc}(n_0,\varepsilon) := \lambda_{\cc}^{\max}+\frac{\varepsilon\lambda_{\cc}^{\max}+\varepsilon^{2}}{n_0+1}.
 \]
\end{lem}
\begin{proof}
One can see that for $n\ge n_0+1$, we have
\begin{equation*}
\lvert 1 - h_{n,k}(\varepsilon)\rvert \le 
 \frac{1}{n_0+1}\left(\lambda_{\cc}^{\max}+\frac{\varepsilon\lambda_{\cc}^{\max}+\varepsilon^{2}}{n_0+1}\right) + o(1) = \frac{E_{\cc}(n_0,\varepsilon)}{n_0+1} + o(1),\quad k \to \infty.
\end{equation*}
Hence, an application of Proposition \ref{Hauke2021lem} gives
\begin{align}
\notag
\prod_{n=n_0+1}^{\lfloor q_k/2\rfloor } h_{n,k}(\varepsilon)&\ge 1-\Biggl(\Bigg\lvert\sum\limits_{n=n_0+1}^{\lfloor q_k/2\rfloor}\Biggl(\bigg(1 - \lambda_k\frac{\{n\cev{\alpha}_k\} - \frac{1}{2}}{n}\bigg)^2 - \frac{\left(\varepsilon + \frac{\lambda_k}{2}\right)^2}{n^2}-1\Biggr)\Bigg\rvert+\frac{E^2_{\cc}(n_0,\varepsilon)}{n_0}\Biggr)-o(1)
\\& \ge 1-\Biggl(2\lambda_{k}\Bigg\lvert\sum\limits_{n=n_0+1}^{\lfloor q_k/2\rfloor}\frac{\{n\cev{\alpha}_k\} - \frac{1}{2}}{n}\Bigg\rvert+\frac{\lambda_{k}^2}{4}\sum\limits_{n=n_0+1}^{\lfloor q_k/2\rfloor}\frac{1}{n^2}+(\varepsilon+\lambda_k/2)^2\sum\limits_{n=n_0+1}^{\lfloor q_k/2\rfloor}\frac{1}{n^2}+\frac{E^2_{\cc}(n_0,\varepsilon)}{n_0}\Biggr)-o(1).
\label{sumsplit}
\end{align}
If any term of \eqref{sumsplit} is unbounded, then \eqref{step_between} holds by the trivial estimate $\prod_{n=n_0+1}^{\lfloor q_k/2\rfloor } h_{n,k}(\alpha,\varepsilon) \ge 0$. If all terms in \eqref{sumsplit} are bounded we use $\lambda_k \le \lambda_{\cc}^{\max}+ o(1)$ to obtain 

\[
\prod_{n=n_0+1}^{\lfloor q_k/2\rfloor } h_{n,k}(\alpha,\varepsilon)\ge 
  1-\Biggl(2\lambda_{\cc}^{\max}\Bigg\lvert\sum\limits_{n=n_0+1}^{\lfloor q_k/2\rfloor}\frac{\frac{1}{2}-\{n\cev{\alpha}_k\}}{n}\Bigg\rvert+\frac{1}{n_0}\biggl(\frac{(\lambda_{\cc}^{\max})^2}{4}+(\varepsilon+0.5\lambda_{\cc}^{\max})^2+E^2_{\cc}(n_0,\varepsilon)\biggr)\Biggr) - o(1).
\]
Therefore, we are left to show that
\begin{equation*}
\left|\sum_{n=n_0+1}^{\lfloor q_k/2\rfloor}\frac{\frac{1}{2} -\left\{n\cev{\alpha}_k\right\}}{n}\right|
\le \left|\sum_{\ell=n_0+1}^{T}\frac{S_{\ell}(\cev{\alpha}_k)}{\ell(\ell +1)}\right| 
 + \frac{9}{2}\frac{1+\log{T}}{T}+o(1), \quad k \to \infty.
 \end{equation*}
Applying summation by parts,
we obtain
\begin{equation*}
\begin{split}
\left|\sum_{n=n_0+1}^{\lfloor q_k/2\rfloor}\frac{\frac{1}{2} -\left\{n\cev{\alpha}_k\right\}}{n}\right|
&\le \left|\sum_{\ell=n_0+1}^{T}\frac{S_{\ell}(\cev{\alpha}_k)}{\ell(\ell +1)}\right| +\frac{|S_{\lfloor q_k/2\rfloor}|}{\lfloor q_k/2\rfloor}
 + \sum_{\ell=T+1}^{\lfloor q_k/2\rfloor}\frac{\lvert S_{\ell}(\cev{\alpha}_k)\rvert}{\ell^2}.
 \end{split}
 \end{equation*}
 We write $\cev{\alpha}_k = [0;b_1^{(k)},b_2^{(k)},\ldots]$ with convergents
 $\frac{p_i^{(k)}}{q_i^{(k)}}$. Since $\alpha \in E_3$, we have
$b_i^{(k)} \le 3$.
Thus Ostrowski's estimate \eqref{Ostr_est} gives for any $n_0+1 \le \ell \le T$ that

\[\lvert S_{\ell}(\cev{\alpha}_k)\rvert \le \frac{9}{2}\log \ell,\]
hence we obtain
 \begin{equation}
 \label{sum_from_Tplus1}
 \sum_{\ell=T+1}^{\lfloor q_k/2\rfloor}\frac{\lvert S_{\ell}(\cev{\alpha}_k)\rvert}{\ell^2}\le
 \frac{9}{2} \sum_{\ell=T+1}^{\infty}\frac{ \log \ell}{\ell^2}\le \frac{9}{2}\frac{1+\log{T}}{T}.
\end{equation}    
 Since $\alpha$ is badly approximable, another application of \eqref{Ostr_est} shows that
\begin{equation*}
\frac{|S_{\lfloor q_k/2\rfloor}|}{\lfloor q_k/2\rfloor} = O_{\alpha}\left(\frac{\log(q_k)}{q_k}\right) = o(1), \quad k \to \infty.
\end{equation*}

\end{proof}
 \begin{figure}[h]
  \centering
  \includegraphics[width=1.\linewidth]{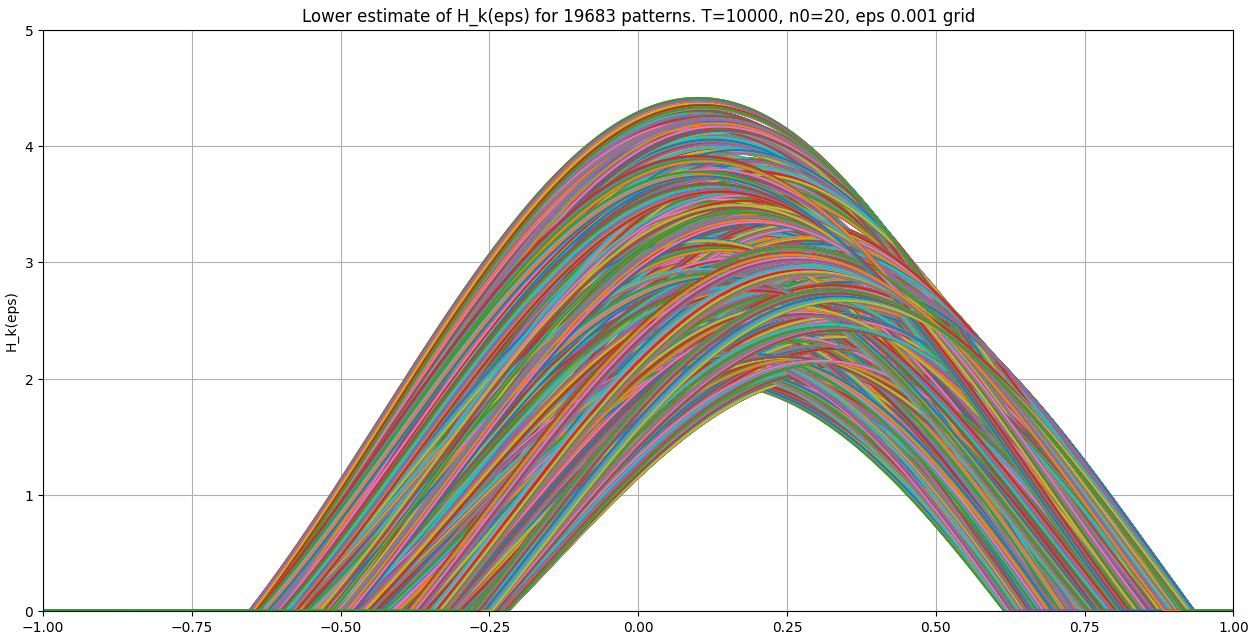}
  \caption{Functions $F_{\cc}(\varepsilon)$ for all $3^9=19683$ patterns $\cc$.}
  \label{fig1}
\end{figure}

Note that the right-hand side of \eqref{step_between} still depends on $\cev{\alpha}_k$ and therefore cannot serve as a function $f_{\cc, \infty}(\varepsilon)$. So in order to finish the construction of $f_{\cc, \infty}(\varepsilon)$ we need to estimate the sum
\begin{equation*}
\left|\sum_{\ell=n_0+1}^{T}\frac{S_{\ell}(\cev{\alpha}_k)}{\ell(\ell +1)}\right|.
\end{equation*}
We use an approach similar to the one established in \cite[Lemma 8]{hauke_badly}: for real numbers $x<y$ define
$$
w_{\min}(\ell,x,y):= \sum_{n =1}^{\ell} \left(1-\{ny\}\right)\cdot\mathds{1}_{[\lfloor nx \rfloor = \lfloor ny \rfloor]}-\frac{1}{2}, \quad w_{\max}(\ell,x,y):= \sum_{n =1}^{\ell} \frac{1}{2} - \left(\{nx\}\cdot\mathds{1}_{[\lfloor nx \rfloor = \lfloor ny \rfloor]}\right).
$$
One can easily see that if $x\le \cev{\alpha}_k\le y$, then
\begin{equation}\label{sandwich}
w_{\min}(\ell,x,y)\le S_{\ell}(\cev{\alpha}_k) \le w_{\max}(\ell,x,y).
\end{equation}
Hence,
\begin{equation}
\label{west}
\left|\sum_{\ell=n_0+1}^{T}\frac{S_{\ell}(\cev{\alpha}_k)}{\ell(\ell +1)}\right|\le\max\left(\left|\sum_{\ell=n_0+1}^{T}\frac{w_{\min}(\ell,x,y)}{\ell(\ell +1)}\right|, \left|\sum_{\ell=n_0+1}^{T}\frac{w_{\max}(\ell,x,y)}{\ell(\ell +1)}\right| \right) =: w_{est}(n_0, T,x,y).
\end{equation}
Choosing $ x= \cev{\cc}_{\min}, y = \cev{\cc}_{\max}$ turns out to give too weak bounds, thus we consider a covering of $(\cev{\cc}_{\min}, \cev{\cc}_{\max})$ with smaller intervals. We do this by an algorithm which receives an integer parameter $m$ that defines the depth (and therefore determines the number of subintervals considered) at which point the algorithm stops searching for improvements. Although a brute-force algorithm could take a uniform grid with a tiny step size, we will present a more efficient approach to only consider small intervals where necessary.

In order to define this algorithm, we introduce some extra notation. Let $B=(b_1,\ldots,b_t)$ be a word in the alphabet $\{1,2,3\}$ and define
\begin{equation*}
cf_{\min}(B)=
\begin{cases}
[0;b_1,\ldots,b_t,\overline{3,1}] \quad\text{if $t$ is even,}\\
[0;b_1,\ldots,b_t,\overline{1,3}] \quad\text{if $t$ is odd},
\end{cases}
\end{equation*}
and
\begin{equation*}
cf_{\max}(B)=
\begin{cases}
[0;b_1,\ldots,b_t,\overline{1,3}] \quad\text{if $t$ is even},\\
[0;b_1,\ldots,b_t,\overline{3,1}] \quad\text{if $t$ is odd}.
\end{cases}
\end{equation*}
Denoting ${\cc}_{left}=(c_4,c_3,c_2,c_1)$, we note that $\cev{\cc}_{\min}=cf_{\min}({\cc}_{left})$, $\cev{\cc}_{\max}=cf_{\max}({\cc}_{left})$. The algorithm gets the sequence $C={\cc}_{left}$ and a threshold parameter $err_{\max}$ as inputs. The initialization value of $err_{\max}$ is 
$$
\max(w_{est}\bigl(n_0,T,\cev{\cc}_{\min},\cev{\cc}_{\min}), w_{est}(n_0,T,\cev{\cc}_{\max},\cev{\cc}_{\max})\bigr)+10^{-6}.
$$
Now the algorithm does the following:
\begin{enumerate}
\item{Set $x=cf_{\min}(C),\ \ y=cf_{\max}(C)$. Calculate $w_{est}(n_0,T,x,y)$.
}
\item{If $w_{est}(n_0,T,x,y)<err_{\max}$, return $w_{est}(n_0,T,x,y)$.}
\item{If $w_{est}(n_0,T,x,y)\ge err_{\max}$ and the length of $C$ is less than $m$, the algorithm gets launched recursively for three longer sequences $(C,1)$,  $(C,2)$, $(C,3)$ and the same value of $err_{\max}$ and returns the maximum of the three outputs.}
\item{If the length of $C$ exceeds $m$, the algorithm terminates and gets launched for the initial input sequence $C={\cc}_{left}$ and increased threshold $err_{\max}=1.05\cdot err_{\max}$.}
\end{enumerate}
 \begin{figure}[h]
  \centering
  \includegraphics[width=1.\linewidth]{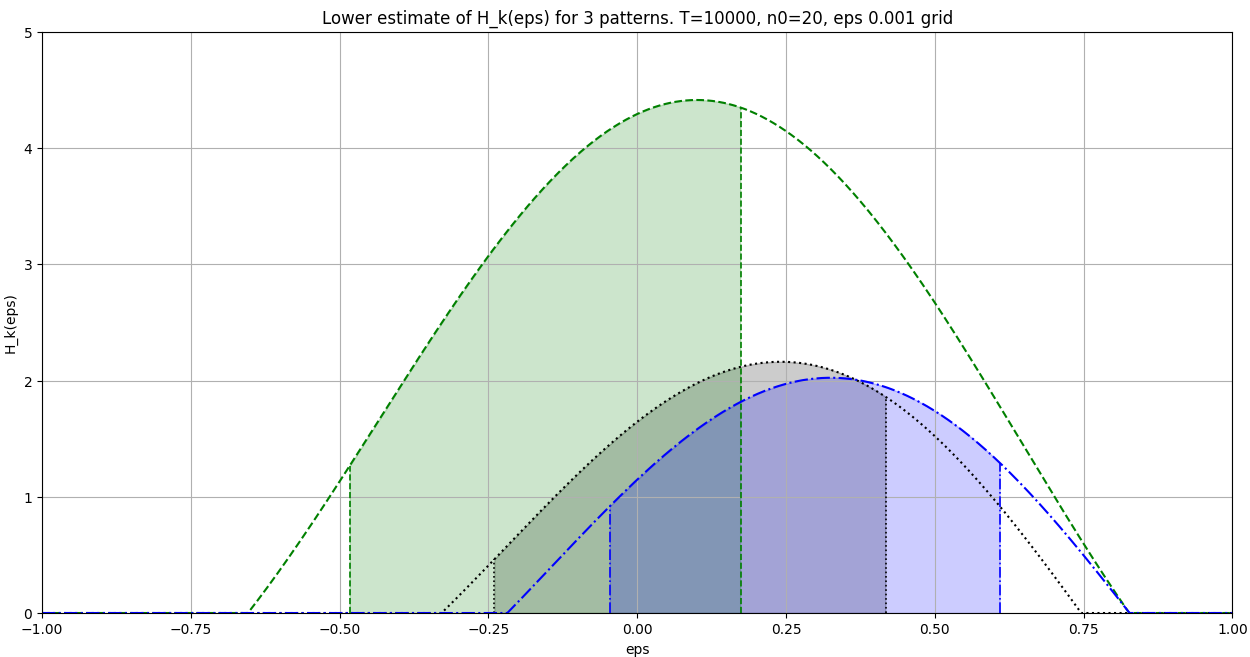}
  \caption{Functions $F_{\cc}(\varepsilon)$ for 3 patterns: $\overline{123123123}$ (black dotted line), $\overline{131313131}$ (green dash-dot line),  and $\overline{313131313}$ (blue dashed line). The range between $\varepsilon_{\cc}^{\min}$ and $\varepsilon_{\cc}^{\max}$ is shaded for each pattern.}
  \label{fig2}
\end{figure}
Writing the output of the above algorithm as $W_{\cc}(n_0,T,m)$, 
it can be verified that
\[ W_{\cc}(n_0,T,m)=\max_{i \in I_m} w_{est}\left(n_0,T,x_i^{(m)},x_{i+1}^{(m)}\right)\]
where the sets $\left\{\left[x_i^{(m)},x_{i+1}^{(m)}\right], i \in I_m\right\}$ cover $(\cev{\cc}_{\min},\cev{\cc}_{\max})$. Thus by \eqref{sandwich} and \eqref{west},
we can deduce that
\begin{equation}
\label{partsum_final_alg}
\left|\sum_{n=n_0+1}^{T}\frac{\frac{1}{2} -\left\{n\cev{\alpha}_k\right\}}{n}\right|
\le W_{\cc}(n_0, T,m)
 +o(1), \quad k \to \infty.
 \end{equation}

We have now all functions defined to finish the proof of Lemma \ref{new_idea_lem} which we repeat here for the reader's convenience.\\

\noindent \textbf{Lemma \ref{new_idea_lem}.} 
    Let $n_0 \ge 10, T \ge n_0+2$ and $m \in \N$ be fixed parameters. For $\cc = (c_1,\ldots,c_9) \in \{1,2,3\}^9$ let
\begin{equation*}
 F_{\cc}(\varepsilon)= F_{\cc}(\varepsilon,n_0,T,m) :=
2\pi(\varepsilon+\lambda_{\cc}^{\min})f_{\cc, \infty}(\varepsilon) \prod_{n=1}^{n_0 } f_{\cc,n}(\varepsilon), \quad \varepsilon\in(\varepsilon_{\cc}^{\min},\varepsilon_{\cc}^{\max}),
\end{equation*}
where
\begin{equation*} 
f_{\cc,\infty}(\varepsilon)=1-\Biggl(2\lambda_{\cc}^{\max}\biggl(W_{\cc}(n_0, T,m)
 + \frac{9}{2}\frac{1+\log{T}}{T}\biggr)+\frac{1}{n_0}\biggl(\frac{(\lambda_{\cc}^{\max})^2}{4}+(\varepsilon+0.5\lambda_{\cc}^{\max})^2+E^2_{\cc}(n_0,\varepsilon)\biggr)\Biggr).
\end{equation*}
Then we have the following:
\begin{itemize}
    \item $F_{\cc}(\varepsilon)$ is pseudo-concave and Lipschitz-continuous on $(\varepsilon_{\cc}^{\min},\varepsilon_{\cc}^{\max})$.
    \item For any $\alpha \in E_3$, let $k_{\ell}(\cc,\alpha):= \{k \in \N: (a_{k-3}(\alpha),a_{k-2}(\alpha),\ldots, a_{k+4}(\alpha),a_{k+5}(\alpha))=\cc\}$.
     Then we have 
\[
\liminf_{\ell \to \infty}H_{k_{\ell}}(\alpha,\varepsilon) \ge F_{\cc}(\varepsilon)
\]
uniformly in $(\varepsilon_{\cc}^{\min},\varepsilon_{\cc}^{\max})$
and $E_3$.
\end{itemize}

\begin{proof}
We will prove that every factor of $F_{\cc}$ is log-concave, which implies log-concavity of $F_{\cc}$, which further implies pseudo-concavity.
Clearly, $2\pi(\varepsilon+\lambda_{\cc}^{\min})$ is a concave function in $\varepsilon$. By Lemma \ref{fcn_lem}, $f_{\cc,n}(\varepsilon)$ is concave and therefore also log-concave. For $f_{\cc,\infty}(\varepsilon)$, we compute the second derivative: We obtain

\[
f_{\cc,\infty}''(\varepsilon) = 
- \frac{2}{n_0} - \frac{1}{n_0}\left(
2E_{\cc}'(n_0,\varepsilon)^2 + 2E_{\cc}(n_0,\varepsilon)E_{\cc}''(n_0,\varepsilon)
\right).
\]

For $\varepsilon \in (\varepsilon_{\cc}^{\min},\varepsilon_{\cc}^{\max})$, we have $\lvert E_{\cc}(n_0,\varepsilon)\rvert \le \lambda_{\max} + \frac{2}{n_0+1}, E''(n_0,\varepsilon) = \frac{2}{n_0+1}$, so for $n_0 \ge 10$ we obtain
$
f_{\cc,\infty}''(\varepsilon) <0$ for any $\varepsilon \in (\varepsilon_{\cc}^{\min},\varepsilon_{\cc}^{\max})$. Thus $f_{\cc,\infty}(\varepsilon)$ is concave and therefore also log-concave. This shows that $F_{\cc}$ is log-concave on $\varepsilon\in(\varepsilon_{\cc}^{\min},\varepsilon_{\cc}^{\max})$. Since every factor of $F_{\cc}$ is Lipschitz-continuous, so is 
$F_{\cc}$ itself.\\
The second statement of Lemma \ref{new_idea_lem} follows from a combination of \eqref{Hk_step1}, \eqref{partsum_final_alg}, and Lemmas \ref{fcn_lem} and \ref{taylor_lem}.
\end{proof}

\section{Case study.}\label{case_study}
In this section, we will show that for all $\cc \in \{1,2,3\}^9$ the functions $F_{\cc} \in\mathcal{F}$ constructed in the previous section satisfy the inequalities (\ref{mainineq_negpert_})-(\ref{mainineq_no_large_negative_}). Due to the huge number of cases and the substantially large parameters $n_0 = 20,T=10.000$, all statements in this section were obtained with computer assistance. We recall that all functions $F_{\cc}$ are pseudo-concave, i.e.
$$
x<\varepsilon<y\implies F_{\cc}(\varepsilon) \ge \min\{F_{\cc}(x),F_{\cc}(y)\}.
$$
We will always assume $k$ to be large enough such that the terms $o(1)$ arising from (\ref{prop8est}) and related quantities defined in Definition/Proposition \ref{prop_lambda_def} become negligible factors in the argument.
Since $F_{\cc}$ is Lipschitz-continuous on the considered range, this factor will also be negligible after an application of $F_{\cc}$. Thus by an abuse of notation, instead of writing e.g.
$$
\lambda_{\cc}^{\min}-o(1) \le \lambda_k\le \lambda_{\cc}^{\max}+o(1)\implies 
F_{\cc}(\lambda_k) \ge \min\{F_{\cc}(\lambda_{\cc}^{\min}),F_{\cc}(\lambda_{\cc}^{\max})\} - o(1),
$$
we will say 
$$
\lambda_{\cc}^{\min} \le \lambda_k \le \lambda_{\cc}^{\max}\implies
F_{\cc}(\lambda_k) \ge \min\{F_{\cc}(\lambda_{\cc}^{\min}),F_{\cc}(\lambda_{\cc}^{\max})\}.
$$
Fix the grid
\begin{equation}
\label{griddef}   
x_i=-1+i\cdot0.001,\quad i=0,1,\ldots, 2000.
\end{equation}
The following statement will be our main tool in the estimation of the functions $F_{\cc_k}(\varepsilon)$ when we know some upper and lower bounds of $\varepsilon$ and the values $F_{\cc}(x_i)$. 
\begin{prop}
\label{monot_prop}
\ \\
\begin{enumerate}
\item{For each $\cc \in \{1,2,3\}^9$ the sequence $(F_{\cc}(x_i))_{i = 0}^{2000}$ has a unique maximum $x_{i_{\max}}>0$. Moreover, for $i < i_{\max}$, the sequence is non-decreasing, and for $i > i_{\max}$ non-increasing.} 
\item{Suppose that $x_i\le\varepsilon<0$. Then $F_{\cc_k}(\varepsilon)\ge F_{\cc_k}(x_i)$.}
\end{enumerate}
\end{prop}
\begin{proof}
The first statement follows from the pseudo-concavity of the functions $F_{\cc}(\varepsilon)$ and is verified with computer assistance by checking the correctness for all $\cc \in \{1,2,3\}^9$. The second statement follows from the first one and, again, the pseudo-concavity property. 
\end{proof}

Using Proposition \ref{monot_prop}, we verify the inequality (\ref{mainineq_negpert_}): for each $\cc \in \{1,2,3\}^9$ we take the maximal $x_i(\cc)$ such that $x_i(\cc)<\varepsilon_{\cc(\alpha,k)}^{\min}$ and show that $\min_{\cc}F_{\cc}(x_i(\cc))>0.35$.

As $0$ belongs to the grid (\ref{griddef}), the inequality (\ref{mainineq_zeropert_}) is verified straightforwardly. Actually, we obtain the sharper estimate
\begin{equation}
\label{mainineq_zeropert_sharp}
 F_{\cc}(0)>1.14671,\quad\forall\cc \in \{1,2,3\}^9.
\end{equation}

\subsection{Positive perturbations}
In this subsection, we obtain lower estimates on $F_{\cc_k}(\varepsilon_{k,t}(N))$ when $\varepsilon_{k,t}(N)>0$. If $t = 0$ and assuming $\varepsilon_{k,t}(N)>0$, we have $0\le \varepsilon_{k,t}(N)\le\lambda_{\cc_k,1}^{\max}$ and therefore
\begin{equation}
\label{mainineq_posoper_t0_sharp}
F_{\cc_k}(\varepsilon_{k,0})\ge\min\bigl(F_{\cc_k}(0), F_{\cc_k}(\lambda_{\cc_k,1}^{\max})\bigr)\ge 1.14671.
\end{equation}
It follows immediately from the definition of $\varepsilon_{k,t}$ and \eqref{deltasum} that
\begin{equation}\label{greater0}
\varepsilon_{k,0}>0 \iff \min\limits_{r>0} r: b_{k+r}>0\ \ \text{is even},
\end{equation}
a fact that we will use several times in the following case study.

If $t=1$, we have $c_5\ge 2$. Using the identity (\ref{deltasum}), we have
\begin{equation}
\label{mainineq_posoper_t1_pert}
\varepsilon_{k,1}(N)\ge q_k\Bigl(\delta_k-(a_{k+2}-1)\delta_{k+1}-a_{k+4}\delta_{k+3}-a_{k+6}\delta_{k+5}-\ldots\Bigr)\ge
q_k\delta_{k+1}\ge \lambda_{\cc_k,1}^{\min}.
\end{equation}
Thus, taking the upper bound from (\ref{eps_min_max_new}), we see that 
\begin{equation}
\label{mainineq_posoper_t1_sharp}
F_{\cc}(\varepsilon_{k,1}(N))\ge\min\bigl(F_{\cc}(\lambda_{\cc_k,1}^{\min}), F_{\cc}(\lambda_{\cc_k,1}^{\max}+\lambda_{\cc_k}^{\max})\bigr)\ge 1.53232.
\end{equation}

If $t=2$, we have $c_5=3$. Similarly to the previous case, we obtain the following bounds:
$$
\lambda_{\cc_k,1}^{\min}+\lambda_{\cc_k}^{\min}\le\varepsilon_{k,2}(N)\le
\lambda_{\cc_k,1}^{\max}+2\lambda_{\cc_k}^{\max}
$$
and deduce that
\begin{equation}
\label{mainineq_posoper_t2_sharp}
F_{\cc_k}(\varepsilon_{k,2}(N))\ge\min\bigl(F_{\cc}(\lambda_{\cc_k,1}^{\min})+\lambda_{\cc_k}^{\min}), F_{\cc}(\lambda_{\cc_k,1}^{\max}+2\lambda_{\cc_k}^{\max})\bigr)\ge 1.2866.
\end{equation}
Thus, the inequality (\ref{mainineq_no_large_positive_}) is verified.

\subsection{Negative perturbations}
 The remaining part of the section will be devoted to the proof of inequality (\ref{mainineq_no_large_negative_}). As we have seen in the previous paragraph, only negative perturbations might cause $F_{\cc_k}(\varepsilon)<1$. However, they will always be overcompensated by other factors arising that are greater than $1$, a phenomenon beautifully illustrated by Aistleitner, Technau, and Zafeiropoulos in the proof of \cite[Lemma 3]{tech_zaf}. In their article, the irrational is always of the form $\alpha=[0;\overline{5}]$. This continued fraction expansion of $\alpha$ means, in our terms, that all the patterns $\cc_k$ coincide, which allows the authors of \cite{tech_zaf} to treat all cases in detail. Our situation is more complicated, which is why we only present the first case 1.1.1.1 in detail. For all other cases, we only provide lower and upper bounds for the corresponding perturbations in a compact form and mention additional arguments if needed.\\

We recall that $\alpha\in\mathfrak{A}(\cc_k,k)$ i.e. $\cc_k=(c_1,c_2,\ldots,c_9)=(a_{k-3}(\alpha),a_{k-2}(\alpha),\ldots,a_{k+5}(\alpha))$. We can assume that $b_{k}>0$, since otherwise, the corresponding factor $H_k$ does not appear in the decomposition. The case distinction will be made by fixing the value of the Ostrowski coefficient for $b_{k+1}$ and if necessary, also the coefficients $b_{k+2},b_{k+3},b_{k+4},b_{k+5},b_{k+6}$.
Note that by the definition of the Ostrowski representation, $b_{k}>0$ implies $b_{k+1}\le a_{k+1}-1$.
\subsubsection*{\LARGE{Case 1. $b_{k+1}>0$.}} 
In this case, we have $a_{k+2}\ge 2$. The case splits into several subcases.
\paragraph*{\Large{Case 1.1. $b_{k+1}=1$.}}
\paragraph*{\large{Case 1.1.1. $b_{k+2}=0$.}}
\paragraph*{\normalsize{Case 1.1.1.1. $b_{k+3}\ge 1$.}}
In this case, we estimate
\begin{equation}
\begin{split}
\label{case1111eq}
0>\varepsilon_{k,0}=q_k\sum\limits_{j=1}^{n-k}(-1)^j b_{k+j}\delta_{k+j}\ge&
q_k\Bigl(-\delta_{k+1}-a_{k+4}\delta_{k+3}-a_{k+6}\delta_{k+5}-\ldots\Bigr)\ge q_k\Bigl(-\delta_{k+1}-\delta_{k+2}\Bigr)\ge\\
\ge&-\lambda_{k,1}-\lambda_{k,2} \ge -\lambda_{\cc_k,1}^{\max}-\lambda_{\cc_k,2}^{\max}.
\end{split}
\end{equation}
Suppose that $x_{i}<-\lambda_{\cc_k,1}^{\max}-\lambda_{\cc_k,2}^{\max}<x_{i+1}$. By Proposition \ref{monot_prop}, 
\begin{equation}
\label{case1111_hk1}
F_{\cc_k}(\varepsilon_{k,0})\ge F_{\cc_k}(x_i).
\end{equation}

The quantity $\varepsilon_{k+1,0}$ is estimated from below as follows:
\begin{equation}
\label{case1111eq2}
\varepsilon_{k+1,0}=q_{k+1}\sum\limits_{j=1}^{n-k-1}(-1)^j b_{k+1+j}\delta_{k+1+j}\ge
q_{k+1}\Bigl(\delta_{k+3}-(a_{k+5}-1)\delta_{k+4}-a_{k+7}\delta_{k+6}-\ldots\Bigr)\ge
q_{k+1}\delta_{k+4}\ge\lambda_{k+1,3}\ge\lambda_{\cc_{k+1},3}^{\min}.
\end{equation}
The upper bound is provided by
\begin{equation}
\label{case1111eq21}
\varepsilon_{k+1,0}\le
q_{k+1}\Bigl(a_{k+4}\delta_{k+3}+a_{k+6}\delta_{k+5}+a_{k+8}\delta_{k+7}+\ldots\Bigr)\le q_{k+1}\delta_{k+2}\le\lambda_{k+1,1}\le\lambda_{\cc_{k+1},1}^{\max}.
\end{equation}

We find maximal $j$ respectively minimal $p$ such that $x_{j}<\lambda_{\cc_{k+1},3}^{\min}<\lambda_{\cc_{k+1},1}^{\max}<x_{p}$. 
As the pattern $\cc_k$ is fixed, we know the first eight elements of the pattern $\cc_{k+1}$. Denote 
$$
\cc^{(i)}=(c_2,c_3,\ldots,c_9,i),\quad i\in\{1,2,3\}.
$$
Thus, using the pseudo-concavity property, we have
\begin{equation}
\label{case1111_hk2}
F_{\cc_{k+1}}(\varepsilon_{k+1,0})\ge \min\bigl(F_{\cc^{(1)}}(x_j), F_{\cc^{(2)}}(x_j), F_{\cc^{(3)}}(x_j), F_{\cc^{(1)}}(x_p), F_{\cc^{(2)}}(x_p), F_{\cc^{(3)}}(x_p)\bigr).
\end{equation}
Combining the estimates (\ref{case1111_hk1}) and (\ref{case1111_hk2}) we obtain that
$$
F_{\cc_k}(\varepsilon_{k,0})F_{\cc_{k+1}}(\varepsilon_{k+1,0})\ge 1.061.
$$
The inequality (\ref{mainineq_no_large_negative_}) is verified in this case. As all following cases are treated in a similar way, we will write the corresponding proofs in less detail. 
\paragraph*{\normalsize{Case 1.1.1.2. $b_{k+3}=0, b_{k+4}=0$}.}
In this case, we obtain
\begin{gather*}
    0>\varepsilon_{k,0}\ge
q_k\Bigl(-\delta_{k+1}-a_{k+6}\delta_{k+5}-a_{k+8}\delta_{k+7}-\ldots\Bigr)\ge-\lambda_{k,1}-\lambda_{k,4}\ge -\lambda_{\cc_{k},1}^{\max}-\lambda_{\cc_{k},4}^{\max},
\\
\lambda_{\cc_{k+1},3}^{\max}\ge 
q_{k+1}\Bigl(a_{k+5}\delta_{k+5}+a_{k+8}\delta_{k+7}+\ldots\Bigr)\ge\varepsilon_{k+1,0}\ge
q_{k+1}\Bigl(-a_{k+7}\delta_{k+6}-a_{k+9}\delta_{k+8}-\ldots\Bigr)\ge-\lambda_{\cc_{k+1},4}^{\max},\\
\implies F_{\cc_k}(\varepsilon_{k,0})F_{\cc_{k+1}}(\varepsilon_{k+1,0})\ge 1.2007.
\end{gather*}

\paragraph*{\normalsize{Case 1.1.1.3. $b_{k+3}=0, b_{k+4}=1, b_{k+5}=0, b_{k+6}\ge 1$}.}
In this case, we obtain
\begin{gather*}
0>\varepsilon_{k,0}\ge
q_k\Bigl(-\delta_{k+1}+\delta_{k+4}+\delta_{k+6}-(a_{k+8}-1)\delta_{k+7}-a_{k+10}\delta_{k+9}-\ldots\Bigr)\ge q_k\Bigl(-\delta_{k+1}+\delta_{k+4}\Bigr)\ge
-\lambda_{\cc_{k},1}^{\max}+\lambda_{\cc_{k},4}^{\min},
\\
0>\varepsilon_{k+1,0}\ge
q_{k+1}\Bigl(-\delta_{k+4}-a_{k+7}\delta_{k+6}-a_{k+9}\delta_{k+8}-\ldots\Bigr)\ge q_k\Bigl(-\delta_{k+4}+\delta_{k+5}\Bigr)\ge-\lambda_{\cc_{k+1},3}^{\max}-\lambda_{\cc_{k+1},4}^{\max},
\\
\implies F_{\cc_k}(\varepsilon_{k,0})F_{\cc_{k+1}}(\varepsilon_{k+1,0})\ge 1.0301.
\end{gather*}

\paragraph*{\normalsize{Case 1.1.1.4. $b_{k+3}=0, b_{k+4}=1, b_{k+5}=0, b_{k+6}=0$}.}
In this case, we obtain
\begin{gather*}
0>\varepsilon_{k,0}\ge
q_k\Bigl(-\delta_{k+1}+\delta_{k+4}-a_{k+8}\delta_{k+7}-a_{k+10}\delta_{k+9}-\ldots\Bigr)\ge\\
\ge q_k\Bigl(-\delta_{k+1}+\delta_{k+4}-\delta_{k+6}\Bigr)
\ge q_k\Bigl(-\delta_{k+1}+0.5\delta_{k+4}\Bigr)
\ge -\lambda_{\cc_{k},1}^{\max}+0.5\lambda_{\cc_{k},4}^{\min},
\\
0\ge\varepsilon_{k+1,0}\ge
q_{k+1}\Bigl(-\delta_{k+4}-a_{k+9}\delta_{k+8}-a_{k+11}\delta_{k+10}-\ldots\Bigr)\ge\\
\ge q_{k+1}\Bigl(-\delta_{k+4}-\delta_{k+7}\Bigr)\ge
q_{k+1}\Bigl(-\delta_{k+4}-0.5\delta_{k+5}\Bigr)\ge
-\lambda_{\cc_{k+1},3}^{\max}-0.5\lambda_{\cc_{k+1},4}^{\max},
\\
\implies F_{\cc_k}(\varepsilon_{k,0})F_{\cc_{k+1}}(\varepsilon_{k+1,0})\ge 1.1424.
\end{gather*}

\paragraph*{\normalsize{Case 1.1.1.5. $b_{k+3}=0, b_{k+4}=1, b_{k+5}\ge 1$}.}
In this case, we obtain
\begin{gather*}
0>\varepsilon_{k,0}\ge
q_k\Bigl(-\delta_{k+1}+\delta_{k+4}-(a_{k+6}-1)\delta_{k+5}-a_{k+8}\delta_{k+7}-\ldots\Bigr)\ge
q_k\Bigl(-\delta_{k+1}+\delta_{k+5}\Bigr)\ge
 -\lambda_{\cc_{k},1}^{\max}+\lambda_{\cc_{k},5}^{\min},
\\
\varepsilon_{k+1,0}\ge
q_{k+1}\Bigl(-\delta_{k+4}+\delta_{k+5}-a_{k+7}\delta_{k+6}-a_{k+9}\delta_{k+8}-\ldots\Bigr)\ge
-q_{k+1}\delta_{k+4}\ge -\lambda_{\cc_{k+1},3}^{\max},
\\
\implies F_{\cc_k}(\varepsilon_{k,0})F_{\cc_{k+1}}(\varepsilon_{k+1,0})\ge 1.1945.
\end{gather*}
\paragraph*{\normalsize{Case 1.1.1.6. $b_{k+3}=0, b_{k+4}=2$.}}
The case condition implies $a_{k+5}\ge2$. We obtain
\begin{gather*}
0>\varepsilon_{k,0}\ge
q_k\Bigl(-\delta_{k+1}+2\delta_{k+4}-(a_{k+6}-1)\delta_{k+5}-a_{k+8}\delta_{k+7}-a_{k+10}\delta_{k+9}-\ldots\Bigr)\ge\\
\ge q_k\Bigl(-\delta_{k+1}+\delta_{k+4}+\delta_{k+5}\Bigr)\ge
-\lambda_{\cc_{k},1}^{\max}+\lambda_{\cc_{k},4}^{\min}+\lambda_{\cc_{k},5}^{\min},
\\
0>\varepsilon_{k+1,0}\ge
 q_{k+1}\Bigl(-2\delta_{k+4}-a_{k+7}\delta_{k+6}-a_{k+9}\delta_{k+8}-\ldots\Bigr)\ge q_{k+1}\Bigl(-2\delta_{k+4}-\delta_{k+5}\Bigr)=\\
 =q_{k+1}\Bigl(-\delta_{k+3}+(a_{k+5}-2)\delta_{k+4}\Bigr)\ge
-\lambda_{\cc_{k+1},3}^{\max}+(a_{k+5}-2)\lambda_{\cc_{k+1},4}^{\min},
\\
\implies F_{\cc_k}(\varepsilon_{k,0})F_{\cc_{k+1}}(\varepsilon_{k+1,0})\ge 1.0183.
\end{gather*}

\paragraph*{\normalsize{Case 1.1.1.7. $b_{k+3}=0, b_{k+4}=3$.}}
The case condition implies $a_{k+5}=3$. We obtain 
\begin{gather*}
0>\varepsilon_{k,0}\ge
q_k\Bigl(-\delta_{k+1}+3\delta_{k+4}-(a_{k+6}-1)\delta_{k+5}-a_{k+8}\delta_{k+7}-\ldots\Bigr)\ge
q_k\Bigl(-\delta_{k+1}+2\delta_{k+4}+\delta_{k+5}\Bigr)\ge
-\lambda_{\cc_{k},1}^{\max}+2\lambda_{\cc_{k},4}^{\min}+\lambda_{\cc_{k},5}^{\min},
\\
0>\varepsilon_{k+1,0}\ge
q_{k+1}\Bigl(-3\delta_{k+4}-a_{k+7}\delta_{k+6}-a_{k+9}\delta_{k+8}-\ldots\Bigr)\ge q_{k+1}\Bigl(-3\delta_{k+4}-\delta_{k+5}\Bigr)=
 q_{k+1}\Bigl(-\delta_{k+3}\Bigr)\ge
-\lambda_{\cc_{k+1},3}^{\max},
\\
F_{\cc_k}(\varepsilon_{k,0})F_{\cc_{k+1}}(\varepsilon_{k+1,0})\ge 1.0189.
\end{gather*}

\paragraph*{\large{Case 1.1.2. $b_{k+2}=1$.}}
In this case, we have $a_{k+3}\ge 2$.
\paragraph*{\normalsize{Case 1.1.2.1. $b_{k+3}=0, b_{k+4}=0, b_{k+5}=0$}.}
We obtain
\begin{gather*}
0>\varepsilon_{k,0}\ge
q_k\Bigl(-\delta_{k+1}+\delta_{k+2}-a_{k+8}\delta_{k+7}-a_{k+10}\delta_{k+9}-\ldots\Bigr)\ge\\
\ge q_k\Bigl(-\delta_{k+1}+\delta_{k+2}-\delta_{k+6}\Bigr)\ge
q_k\Bigl(-\delta_{k+1}+\delta_{k+2}-0.5\delta_{k+4}\Bigr)\ge
\lambda_{\cc_{k},1}^{\max}+\lambda_{\cc_{k},2}^{\min}-0.5\lambda_{\cc_{k},4}^{\max},
\\
0>\varepsilon_{k+1,0}\ge
q_{k+1}\Bigl(-\delta_{k+2}-a_{k+7}\delta_{k+6}-a_{k+9}\delta_{k+8}-\ldots\Bigr)\ge
q_{k+1}\Bigl(-\delta_{k+2}-\delta_{k+5}\Bigr)\ge
-\lambda_{\cc_{k+1},1}^{\max}-\lambda_{\cc_{k+1},4}^{\max},
\\
\implies\max\bigl(F_{\cc_k}(\alpha, \varepsilon_{k,0}), F_{\cc_k}(\alpha, \varepsilon_{k,0})F_{\cc_{k+1}}(\alpha, \varepsilon_{k+1,0})\bigr)\ge 1.0046.
\end{gather*}

\paragraph*{\normalsize{Case 1.1.2.2. $b_{k+3}=0, b_{k+4}=0, b_{k+5}\ge 1$}.}
We obtain
\begin{gather*}
0>\varepsilon_{k,0}\ge
q_k\Bigl(-\delta_{k+1}+\delta_{k+2}-a_{k+6}\delta_{k+5}-a_{k+8}\delta_{k+7}-\ldots\Bigr)\ge
q_k\Bigl(-\delta_{k+1}+\delta_{k+2}-\delta_{k+4}\Bigr)\ge
-\lambda_{\cc_{k},1}^{\max}+\lambda_{\cc_{k},2}^{\min}-\lambda_{\cc_{k},4}^{\max},
\\
0>\varepsilon_{k+1,0}\ge
q_{k+1}\Bigl(-\delta_{k+2}+\delta_{k+5}-(a_{k+7}-1)\delta_{k+6}-a_{k+9}\delta_{k+8}-\ldots\Bigr)\ge
-q_{k+1}\delta_{k+2}\ge-\lambda_{\cc_{k+1},1}^{\max},
\\
\implies\max\bigl(F_{\cc_k}(\alpha, \varepsilon_{k,0}), F_{\cc_k}(\alpha, \varepsilon_{k,0})F_{\cc_{k+1}}(\alpha, \varepsilon_{k+1,0})\bigr)\ge 1.071.
\end{gather*}

\paragraph*{\normalsize{Case 1.1.2.3. $b_{k+3}=0, b_{k+4}\ge 1$}.}
We obtain
\begin{gather*}
0>\varepsilon_{k,0}\ge
q_k\Bigl(-\delta_{k+1}+\delta_{k+2}+\delta_{k+4}-(a_{k+6}-1)\delta_{k+5}-a_{k+8}\delta_{k+5}-\ldots\Bigr)\ge\\
\ge q_k\Bigl(-\delta_{k+1}+\delta_{k+2}+\delta_{k+5}\Bigr)\ge
-\lambda_{\cc_{k},1}^{\max}+\lambda_{\cc_{k},2}^{\min}+
\lambda_{\cc_{k},5}^{\min},
\\
0>\varepsilon_{k+1,0}\ge
q_{k+1}\Bigl(-\delta_{k+2}-\delta_{k+4}-a_{k+7}\delta_{k+6}-a_{k+9}\delta_{k+8}-\ldots\Bigr)\ge\\
\ge q_{k+1}\Bigl(-\delta_{k+2}-\delta_{k+4}-\delta_{k+5}\Bigr)\ge
-\lambda_{\cc_{k+1},1}^{\max}-\lambda_{\cc_{k+1},3}^{\max}-\lambda_{\cc_{k+1},4}^{\max}.
\end{gather*}
For the quantity $\varepsilon_{k+2,0}$, we use \eqref{greater0} and inequality (\ref{mainineq_posoper_t0_sharp}). Thus, we obtain
$$
\max\bigl(F_{\cc_k}(\alpha,\ \varepsilon_{k,0}), F_{\cc_k}(\alpha, \varepsilon_{k,0})F_{\cc_{k+1}}(\alpha, \varepsilon_{k+1,0})
F_{\cc_{k+2}}(\alpha, \varepsilon_{k+2,0})\bigr)\ge 1.02901.
$$

\paragraph*{\normalsize{Case 1.1.2.4. $b_{k+3}\ge 1$}.}
In this case, one can see that $a_{k+4}\ge 2$. We obtain
\begin{gather*}
0>\varepsilon_{k,0}\ge
q_k\Bigl(-\delta_{k+1}+\delta_{k+2}-(a_{k+4}-1)\delta_{k+3}-a_{k+6}\delta_{k+5}-a_{k+8}\delta_{k+7}-\ldots\Bigr)\ge
q_k\Bigl(-\delta_{k+1}+\delta_{k+3}\Bigr)\ge
-\lambda_{\cc_{k},1}^{\max}+\lambda_{\cc_{k},3}^{\min},
\\
0>\varepsilon_{k+1,0}\ge
q_{k+1}\Bigl(-\delta_{k+2}+\delta_{k+3}-(a_{k+5}-1)\delta_{k+4}-a_{k+7}\delta_{k+6}-\ldots\Bigr)\ge
q_{k+1}\Bigl(-\delta_{k+2}+\delta_{k+4}\Bigr)\ge
-\lambda_{\cc_{k+1},1}^{\max}+\lambda_{\cc_{k+1},3}^{\min},
\\
\implies \max\bigl(F_{\cc_k}(\alpha, \varepsilon_{k,0}), F_{\cc_k}(\alpha, \varepsilon_{k,0})F_{\cc_{k+1}}(\alpha, \varepsilon_{k+1,0})\bigr)\ge 1.0705.
\end{gather*}
\paragraph*{\large{Case 1.1.3. $b_{k+2}=2$.}}
In this case, we have $a_{k+3}=3$.

\paragraph*{\normalsize{Case 1.1.3.1. $b_{k+3}=0, b_{k+4}=0, b_{k+5}=0, b_{k+6}=0$}.}
We obtain
\begin{gather*}
0>\varepsilon_{k,0}\ge
q_k\Bigl(-\delta_{k+1}+2\delta_{k+2}-a_{k+8}\delta_{k+7}-a_{k+10}\delta_{k+9}-\ldots\Bigr)\ge
q_k\Bigl(-\delta_{k+1}+2\delta_{k+2}-\delta_{k+6}\Bigr)\ge\\
\ge q_k\Bigl(-\delta_{k+1}+2\delta_{k+2}-0.5\delta_{k+4}\Bigr)\ge
-\lambda_{\cc_{k},1}^{\max}+2\lambda_{\cc_{k},2}^{\min}-
0.5\lambda_{\cc_{k},4}^{\max},
\\
0>\varepsilon_{k+1,0}\ge
q_{k+1}\Bigl(-2\delta_{k+2}-a_{k+9}\delta_{k+8}-a_{k+11}\delta_{k+10}-\ldots\Bigr)\ge
q_{k+1}\Bigl(-2\delta_{k+2}-\delta_{k+7}\Bigr)\ge\\
\ge q_{k+1}\Bigl(-2\delta_{k+2}-0.5\delta_{k+5}\Bigr)\ge
-2\lambda_{\cc_{k+1},1}^{\max}-0.5\lambda_{\cc_{k+1},4}^{\max},
\\
\implies \max\bigl(F_{\cc_k}(\alpha, \varepsilon_{k,0}), F_{\cc_k}(\alpha, \varepsilon_{k,0})F_{\cc_{k+1}}(\alpha, \varepsilon_{k+1,0})\bigr)\ge 1.03587.
\end{gather*}

\paragraph*{\normalsize{Case 1.1.3.2. $b_{k+3}=0, b_{k+4}=0, b_{k+5}=0, b_{k+6}\ge 1$}.}
We obtain
\begin{gather*}
0>\varepsilon_{k,0}\ge
q_k\Bigl(-\delta_{k+1}+2\delta_{k+2}+\delta_{k+6}-(a_{k+8}-1)\delta_{k+7}-a_{k+10}\delta_{k+9}-\ldots\Bigr)\ge
q_k\Bigl(-\delta_{k+1}+2\delta_{k+2}\Bigr)\ge
-\lambda_{\cc_{k},1}^{\max}+2\lambda_{\cc_{k},2}^{\min},
\\
0>\varepsilon_{k+1,0}\ge
q_{k+1}\Bigl(-2\delta_{k+2}-a_{k+7}\delta_{k+6}-
-a_{k+9}\delta_{k+8}-\ldots\Bigr)\ge
q_{k+1}\Bigl(-2\delta_{k+2}-\delta_{k+5}\Bigr)\ge
-2\lambda_{\cc_{k+1},1}^{\max}-\lambda_{\cc_{k+1},4}^{\max},
\\
\implies \max\bigl(F_{\cc_k}(\alpha, \varepsilon_{k,0}), F_{\cc_k}(\alpha, \varepsilon_{k,0})F_{\cc_{k+1}}(\alpha, \varepsilon_{k+1,0})\bigr)\ge 1.0139.
\end{gather*}

\paragraph*{\normalsize{Case 1.1.3.3. $b_{k+3}=0, b_{k+4}=0, b_{k+5}\ge 1$}.}
We obtain
\begin{gather*}
0>\varepsilon_{k,0}\ge
q_k\Bigl(-\delta_{k+1}+2\delta_{k+2}-a_{k+6}\delta_{k+5}-a_{k+8}\delta_{k+7}-\ldots\Bigr)\ge
q_k\Bigl(-\delta_{k+1}+2\delta_{k+2}-\delta_{k+4}\Bigr)\ge
-\lambda_{\cc_{k},1}^{\max}+2\lambda_{\cc_{k},2}^{\min}-
\lambda_{\cc_{k},4}^{\max},
\\
0>\varepsilon_{k+1,0}\ge
q_{k+1}\Bigl(-2\delta_{k+2}+\delta_{k+5}-a_{k+7}\delta_{k+6}
-a_{k+9}\delta_{k+8}-\ldots\Bigr)\ge
-2q_{k+1}\delta_{k+2}\ge
-2\lambda_{\cc_{k+1},1}^{\max},
\\
\implies \max\bigl(F_{\cc_k}(\alpha, \varepsilon_{k,0}), F_{\cc_k}(\alpha, \varepsilon_{k,0})F_{\cc_{k+1}}(\alpha, \varepsilon_{k+1,0})\bigr)\ge 1.0466.
\end{gather*}

\paragraph*{\normalsize{Case 1.1.3.4. $b_{k+3}=0, b_{k+4}=1, b_{k+5}=0$}.}
We obtain
\begin{gather*}
0>\varepsilon_{k,0}\ge
q_k\Bigl(-\delta_{k+1}+2\delta_{k+2}+\delta_{k+4}-a_{k+8}\delta_{k+7}-a_{k+10}\delta_{k+9}-\ldots\Bigr)\ge\\
\ge q_k\Bigl(-\delta_{k+1}+2\delta_{k+2}+\delta_{k+4}-\delta_{k+6}\Bigr)\ge
q_k\Bigl(-\delta_{k+1}+2\delta_{k+2}+0.5\delta_{k+4}\Bigr)\ge
-\lambda_{\cc_{k},1}^{\max}+2\lambda_{\cc_{k},2}^{\min}+
0.5\lambda_{\cc_{k},4}^{\min},
\\
0>\varepsilon_{k+1,0}\ge
q_{k+1}\Bigl(-2\delta_{k+2}-\delta_{k+4}-a_{k+7}\delta_{k+6}-
-a_{k+9}\delta_{k+8}-\ldots\Bigr)\ge\\
\ge q_{k+1}\Bigl(-2\delta_{k+2}-\delta_{k+4}-\delta_{k+5}\Bigr)\ge
-2\lambda_{\cc_{k+1},1}^{\max}-\lambda_{\cc_{k+1},3}^{\max}-
\lambda_{\cc_{k+1},4}^{\max},
\\
\implies \max\bigl(F_{\cc_k}(\alpha, \varepsilon_{k,0}), F_{\cc_k}(\alpha, \varepsilon_{k,0})F_{\cc_{k+1}}(\alpha, \varepsilon_{k+1,0})\bigr)\ge 1.0032.
\end{gather*}

\paragraph*{\normalsize{Case 1.1.3.5. $b_{k+3}=0, b_{k+4}=1, b_{k+5}\ge 1$}.}
We obtain
\begin{gather*}
0>\varepsilon_{k,0}\ge
q_k\Bigl(-\delta_{k+1}+2\delta_{k+2}+\delta_{k+4}-(a_{k+6}-1)\delta_{k+5}-a_{k+8}\delta_{k+7}-\ldots\Bigr)\ge\\
\ge q_k\Bigl(-\delta_{k+1}+2\delta_{k+2}+\delta_{k+5}\Bigr)\ge
-\lambda_{\cc_{k},1}^{\max}+2\lambda_{\cc_{k},2}^{\min}+
\lambda_{\cc_{k},5}^{\min},
\\
0>\varepsilon_{k+1,0}\ge
q_{k+1}\Bigl(-2\delta_{k+2}-\delta_{k+4}+\delta_{k+5}-a_{k+7}\delta_{k+6}
-a_{k+9}\delta_{k+8}-\ldots\Bigr)\ge
q_{k+1}\Bigl(-2\delta_{k+2}-\delta_{k+4}\Bigr)\ge
-2\lambda_{\cc_{k+1},1}^{\max}-\lambda_{\cc_{k+1},3}^{\max},
\\
\implies \max\bigl(F_{\cc_k}(\alpha, \varepsilon_{k,0}), F_{\cc_k}(\alpha, \varepsilon_{k,0})F_{\cc_{k+1}}(\alpha, \varepsilon_{k+1,0})\bigr)\ge 1.0084.
\end{gather*}

\paragraph*{\normalsize{Case 1.1.3.6. $b_{k+3}=0, b_{k+4}\ge 2$}.}
In this case, we have $a_{k+5}\ge 2$. We obtain
\begin{gather*}
0>\varepsilon_{k,0}\ge
q_k\Bigl(-\delta_{k+1}+2\delta_{k+2}+2\delta_{k+4}-(a_{k+6}-1)\delta_{k+5}-a_{k+8}\delta_{k+7}-\ldots\Bigr)\ge\\
\ge q_k\Bigl(-\delta_{k+1}+2\delta_{k+2}+\delta_{k+4}+\delta_{k+5}\Bigr)\ge
-\lambda_{\cc_{k},1}^{\max}+2\lambda_{\cc_{k},2}^{\min}+
\lambda_{\cc_{k},4}^{\min}+\lambda_{\cc_{k},5}^{\min},
\\
\implies F_{\cc_k}(\alpha, \varepsilon_{k,0})>1.0179.
\end{gather*}

\paragraph*{\normalsize{Case 1.1.3.7. $b_{k+3}\ge 1$}.}
We obtain
\begin{gather*}
0>\varepsilon_{k,0}\ge
q_k\Bigl(-\delta_{k+1}+2\delta_{k+2}-(a_{k+4}-1)\delta_{k+3}-a_{k+6}\delta_{k+5}-a_{k+8}\delta_{k+7}-\ldots\Bigr)\ge\\
\ge q_k\Bigl(-\delta_{k+1}+\delta_{k+2}+\delta_{k+3}\Bigr)\ge
-\lambda_{\cc_{k},1}^{\max}+\lambda_{\cc_{k},2}^{\min}+
\lambda_{\cc_{k},3}^{\min},
\\
0>\varepsilon_{k+1,0}\ge
q_{k+1}\Bigl(-2\delta_{k+2}+\delta_{k+3}-(a_{k+5}-1)\delta_{k+4}-a_{k+7}\delta_{k+6}-\ldots\Bigr)\ge
q_{k+1}\Bigl(-2\delta_{k+2}+\delta_{k+4}\Bigr)\ge
-2\lambda_{\cc_{k+1},1}^{\max}+\lambda_{\cc_{k+1},3}^{\min},
\\
\implies \max\bigl(F_{\cc_k}(\alpha, \varepsilon_{k,0}), F_{\cc_k}(\alpha, \varepsilon_{k,0})F_{\cc_{k+1}}(\alpha, \varepsilon_{k+1,0})\bigr)\ge 1.0325.
\end{gather*}

\paragraph*{\Large{Case 1.2. $b_{k+1}=2$.}}
\ \\
In this case, we have $a_{k+2}=3$.
We use the following lower estimate of $\varepsilon_{k,0}$.
\begin{equation*}
\begin{split}
\label{case12eq}
\varepsilon_{k,0}\ge
q_k\Bigl(-2\delta_{k+1}-a_{k+4}\delta_{k+3}-a_{k+6}\delta_{k+5}-\ldots\Bigr)\ge
-q_k\Bigl(-2\delta_{k+1}-\delta_{k+2}\Bigr)\ge
-2\lambda_{\cc_{k+1},1}^{\max}-\lambda_{\cc_{k+1},2}^{\max}.
\end{split}
\end{equation*}
The perturbation $\varepsilon_{k+1,1}$ is estimated by (\ref{mainineq_posoper_t1_sharp}). We obtain
\begin{equation*}
F_{\cc_k}(\alpha, \varepsilon_{k,0})F_{\cc_{k+1}}(\alpha, \varepsilon_{k+1,1})\ge 1.53477.
\end{equation*}

\subsubsection*{\LARGE{Case 2. $b_{k+1}=0$.}}
\paragraph*{\Large{Case 2.1. $b_{k+2}\ge 1$.}}
\ \\
In this case, it follows immediately from (\ref{mainineq_posoper_t0_sharp}) and \eqref{greater0} that
$F_{\cc_k}(\alpha, \varepsilon_{k,0})>1.14671$.

\paragraph*{\Large{Case 2.2. $b_{k+2}=0$.}}
\paragraph*{\large{Case 2.2.1. $b_{k+3}=0$.}}

\paragraph*{\normalsize{Case 2.2.1.1. $b_{k+4}=0, b_{k+5}=0$}.} 
We obtain
\begin{gather*}
\varepsilon_{k,0}\ge
q_k\Bigl(-a_{k+8}\delta_{k+7}-a_{k+10}\delta_{k+9}-\ldots\Bigr)\ge
-q_k\delta_{k+6}\ge -0.5 q_k\delta_{k+4}\ge-0.5\lambda_{\cc_{k},4}^{\max},
\\
\implies F_{\cc_k}(\alpha, \varepsilon_{k,0})>1.0841.
\end{gather*}

\paragraph*{\normalsize{Case 2.2.1.2. $b_{k+4}=0, b_{k+5}\ge1$, \textnormal{minimal $r>5$ such that 
$b_{k+r}\ge1$ is \textbf{odd}}}.}
We obtain
\begin{equation*}
0>\varepsilon_{k,0}\ge
q_k\Bigl(-a_{k+6}\delta_{k+5}-a_{k+8}\delta_{k+7}-\ldots\Bigr)\ge
-q_k\delta_{k+4}\ge -\lambda_{\cc_{k},4}^{\max}.
\end{equation*}
On the other hand, by \eqref{greater0}, we see that $\varepsilon_{k+5,0}\ge 0$. Using the inequality (\ref{mainineq_posoper_t0_sharp}), we obtain
$$
F_{\cc_k}(\alpha, \varepsilon_{k,0})F_{\cc_{k+5}}(\alpha, \varepsilon_{k+5,0})>1.1296.
$$

\paragraph*{\normalsize{Case 2.2.1.3. $b_{k+4}=0, b_{k+5}\ge1$, \textnormal{minimal $r>5$ such that 
$b_{k+r}\ge1$ is \textbf{even}}}.}
We use the following lower estimate of $\varepsilon_{k,0}$ in this case
\begin{equation*}
0>\varepsilon_{k,0}\ge
q_k\Bigl(-b_{k+5}\delta_{k+5}+b_{k+r}\delta_{k+r}-\ldots\Bigr)\ge
-b_{k+5}q_k\delta_{k+5}\ge -b_{k+5}\lambda_{\cc_{k},5}^{\max}.
\end{equation*}
Using (\ref{mainineq_posoper_t0_sharp}), we see that
$$
F_{\cc_k}(\alpha, \varepsilon_{k,0})\prod\limits_{t=1}^{b_{k+5}-1}F_{\cc_{k+5}}(\alpha, \varepsilon_{k+5,t})\ge F_{\cc_k}(\alpha, \varepsilon_{k,0})1.14671^{(b_{k+5}-1)}> 1.0596.
$$
\paragraph*{\normalsize{Case 2.2.1.4 $b_{k+4}\ge1$}.}
Follows from (\ref{mainineq_posoper_t0_sharp}) and \eqref{greater0}.

\paragraph*{\large{Case 2.2.2. $b_{k+3}=1, b_{k+4}=0$}.}

\paragraph*{\normalsize{Case 2.2.2.1. $b_{k+5}=0$, \textnormal{minimal $r>5$ such that 
$b_{k+r}\ge1$ is \textbf{odd}}}.}
We obtain
\begin{equation*}
\begin{split}
0>\varepsilon_{k,0}&\ge
q_k\Bigl(-\delta_{k+3}-b_{k+r}\delta_{k+r}-\ldots\Bigr)\ge
q_k\Bigl(-\delta_{k+3}-a_{k+8}\delta_{k+7}-a_{k+10}\delta_{k+9}-\ldots\Bigr)\ge \\&\ge
q_k\Bigl(-\delta_{k+3}-\delta_{k+6}\Bigr)\ge
q_k\Bigl(-\delta_{k+3}-0.5\delta_{k+4}\Bigr)\ge
-\lambda_{\cc_{k},3}^{\max}-0.5\lambda_{\cc_{k},4}^{\max}.
\end{split}
\end{equation*}
By \eqref{greater0}, we have $\varepsilon_{k+3,0}>0$. As the pattern $\cc_k$ is fixed, we know the first six elements of the pattern $\cc_{k+3}$. Denote 
$$
\cc^{(i,j,k)}=(c_4,c_5,\ldots,c_9,i,j,k),\quad i,j,k\in\{1,2,3\}.
$$
We have
\begin{equation*}
F_{\cc_{k+3}}(\varepsilon_{k+3,0})\ge F_{\cc_{k+3}}(0)\ge \min\limits_{1\le i,j,k\le 3}\bigl(F_{\cc^{(i,j,k)}}(0)\bigr)
\end{equation*}
and finally, obtain
$$
F_{\cc_k}(\alpha, \varepsilon_{k,0})F_{\cc_{k+3}}(\alpha, \varepsilon_{k+3,0})\ge 1.432.
$$

\paragraph*{\normalsize{Case 2.2.2.2. $b_{k+5}=0$, \textnormal{minimal $r>5$ such that 
$b_{k+r}\ge1$ is \textbf{even}}}.}
We obtain
\begin{gather*}
0>\varepsilon_{k,0}\ge
q_k\Bigl(-\delta_{k+3}+b_{k+r}\delta_{k+r}-\ldots\Bigr)\ge
-q_k\delta_{k+3}\ge
-\lambda_{\cc_{k},3}^{\max},
\\
0>\varepsilon_{k+3,0}\ge
q_{k+3}\Bigl(-a_{k+7}\delta_{k+6}-a_{k+9}\delta_{k+8}-\ldots\Bigr)\ge
-q_{k+3}\delta_{k+5}\ge -\lambda_{\cc_{k+3},2}^{\max},
\\
\implies \max\bigl(F_{\cc_k}(\alpha, \varepsilon_{k,0}),F_{\cc_k}(\alpha, \varepsilon_{k,0})F_{\cc_{k+3}}(\alpha, \varepsilon_{k+3,0})\bigr)\ge 1.0535.
\end{gather*}

\paragraph*{\normalsize{Case 2.2.2.3. $b_{k+5}\ge 1$}.}
We obtain
\begin{gather*}
0>\varepsilon_{k,0}\ge
q_k\Bigl(-\delta_{k+3}-a_{k+6}\delta_{k+5}-a_{k+8}\delta_{k+7}-\ldots\Bigr)\ge
q_k\Bigl(-\delta_{k+3}-\delta_{k+4}\Bigr)\ge -\lambda_{\cc_{k},3}^{\max}-\lambda_{\cc_{k},4}^{\max},
\\
\lambda_{\cc_{k+3},1}^{\max}\ge q_{k+3}\delta_{k+4}\ge\varepsilon_{k+3,0}\ge
 q_{k+3}\Bigl(\delta_{k+5}-(a_{k+7}-1)\delta_{k+6}-a_{k+9}\delta_{k+8}-\ldots\Bigr)\ge
q_{k+3}\delta_{k+6}\ge \lambda_{\cc_{k+3},3}^{\min},
\\
\implies F_{\cc_k}(\alpha, \varepsilon_{k,0})F_{\cc_{k+3}}(\alpha, \varepsilon_{k+3,0})\ge 1.434.
\end{gather*}

\paragraph*{\large{Case 2.2.3. $b_{k+3}=1, b_{k+4}\ge 1$}.}
In this case, we have $a_{k+5}\ge 2$.

\paragraph*{\normalsize{Case 2.2.3.1. $b_{k+4}=1, b_{k+5}=0, b_{k+6}=0$}}
We obtain
\begin{gather*}
0>\varepsilon_{k,0}\ge
q_k\Bigl(-\delta_{k+3}+\delta_{k+4}-a_{k+8}\delta_{k+7}-a_{k+10}\delta_{k+9}-\ldots\Bigr)\ge\\
\ge q_k\Bigl(-\delta_{k+3}+\delta_{k+4}-\delta_{k+6}\Bigr)\ge
q_k\Bigl(-\delta_{k+3}+0.5\delta_{k+4}\Bigr)\ge
-\lambda_{\cc_{k},3}^{\max}+0.5\lambda_{\cc_{k},4}^{\min},
\\
0>\varepsilon_{k+3,0}\ge
q_{k+3}\Bigl(-\delta_{k+4}-a_{k+9}\delta_{k+8}-a_{k+11}\delta_{k+10}-\ldots\Bigr)\ge\\
\ge q_{k+3}\Bigl(-\delta_{k+4}-\delta_{k+7}\Bigr)\ge
q_{k+3}\Bigl(-\delta_{k+4}-0.5\delta_{k+5}\Bigr)\ge
-\lambda_{\cc_{k+3},1}^{\max}-0.5\lambda_{\cc_{k+3},2}^{\max},
\\
\implies \max\bigl(F_{\cc_k}(\alpha, \varepsilon_{k,0}), F_{\cc_k}(\alpha, \varepsilon_{k,0})F_{\cc_{k+3}}(\alpha, \varepsilon_{k+3,0})\bigr)\ge 1.073.
\end{gather*}

\paragraph*{\normalsize{Case 2.2.3.2. $b_{k+4}=1, b_{k+5}=0, b_{k+6}\ge 1$}.}
We obtain
\begin{gather*}
0>\varepsilon_{k,0}\ge
q_k\Bigl(-\delta_{k+3}+\delta_{k+4}+\delta_{k+6}-(a_{k+8}-1)\delta_{k+7}-a_{k+10}\delta_{k+9}-\ldots\Bigr)\ge
-q_k\Bigl(-\delta_{k+3}+\delta_{k+4}\Bigr)\ge
-\lambda_{\cc_{k},3}^{\max}+\lambda_{\cc_{k},4}^{\min},
\\
0>\varepsilon_{k+3,0}\ge
q_{k+3}\Bigl(-\delta_{k+4}-a_{k+7}\delta_{k+6}-a_{k+9}\delta_{k+8}-\ldots\Bigr)\ge
q_{k+3}\Bigl(-\delta_{k+4}-\delta_{k+5}\Bigr)\ge -\lambda_{\cc_{k+3},1}^{\max}-\lambda_{\cc_{k+3},2}^{\max},
\\
\implies \max\bigl(F_{\cc_k}(\alpha, \varepsilon_{k,0}), F_{\cc_k}(\alpha, \varepsilon_{k,0})F_{\cc_{k+3}}(\alpha, \varepsilon_{k+3,0})\bigr)\ge 1.063.
\end{gather*}

\paragraph*{\normalsize{Case 2.2.3.3. $b_{k+4}=1, b_{k+5}\ge 1$}.}
We obtain
\begin{gather*}
0>\varepsilon_{k,0}\ge
q_k\Bigl(-\delta_{k+3}+\delta_{k+4}-(a_{k+6}-1)\delta_{k+5}-a_{k+8}\delta_{k+7}-\ldots\Bigr)\ge
q_k\Bigl(-\delta_{k+3}+\delta_{k+5}\Bigr)\ge
-\lambda_{\cc_{k},3}^{\max}+\lambda_{\cc_{k},5}^{\min},
\\
0>\varepsilon_{k+3,0}\ge
q_{k+3}\Bigl(-\delta_{k+4}+\delta_{k+5}-(a_{k+7}-1)\delta_{k+6}-a_{k+9}\delta_{k+8}-\ldots\Bigr)\ge
q_{k+3}\Bigl(-\delta_{k+4}+\delta_{k+6}\Bigr)\ge
-\lambda_{\cc_{k+3},1}^{\max}+\lambda_{\cc_{k+3},3}^{\min},
\\
\implies \max\bigl(F_{\cc_k}(\alpha, \varepsilon_{k,0}), F_{\cc_k}(\alpha, \varepsilon_{k,0})F_{\cc_{k+3}}(\alpha, \varepsilon_{k+3,0})\bigr)\ge 1.0622.
\end{gather*}

\paragraph*{\normalsize{Case 2.2.3.4. $b_{k+4}=2, b_{k+5}=0, b_{k+6}=0$}.}
\ \\
Note that in cases 2.2.3.4-2.2.3.6 we have $a_{k+5}=3$. We obtain
\begin{gather*}
0>\varepsilon_{k,0}\ge
q_k\Bigl(-\delta_{k+3}+2\delta_{k+4}-a_{k+8}\delta_{k+7}-a_{k+10}\delta_{k+9}-\ldots\Bigr)\ge\\
\ge q_k\Bigl(-\delta_{k+3}+2\delta_{k+4}-\delta_{k+6}\Bigr)\ge
q_k\Bigl(-\delta_{k+3}+1.5\delta_{k+4}\Bigr)\ge
-\lambda_{\cc_{k},3}^{\max}+1.5\lambda_{\cc_{k},4}^{\min},
\\
0>\varepsilon_{k+3,0}\ge
q_{k+3}\Bigl(-2\delta_{k+4}-a_{k+9}\delta_{k+8}-a_{k+11}\delta_{k+10}-\ldots\Bigr)\ge\\
\ge q_{k+3}\Bigl(-2\delta_{k+4}-\delta_{k+7}\Bigr)\ge
q_{k+3}\Bigl(-2\delta_{k+4}-0.5\delta_{k+5}\Bigr)\ge
-2\lambda_{\cc_{k+3},1}^{\max}+0.5\lambda_{\cc_{k+3},2}^{\min},
\\
\implies \max\bigl(F_{\cc_k}(\alpha, \varepsilon_{k,0}), F_{\cc_k}(\alpha, \varepsilon_{k,0})F_{\cc_{k+3}}(\alpha, \varepsilon_{k+3,0})\bigr)\ge 1.0659.
\end{gather*}

\paragraph*{\normalsize{Case 2.2.3.5. $b_{k+4}=2, b_{k+5}=0, b_{k+6}\ge 1$}.}
We obtain
\begin{gather*}
0>\varepsilon_{k,0}\ge
q_k\Bigl(-\delta_{k+3}+2\delta_{k+4}+\delta_{k+6}-(a_{k+8}-1)\delta_{k+7}-a_{k+10}\delta_{k+9}-\ldots\Bigr)\ge
q_k\Bigl(-\delta_{k+3}+2\delta_{k+4}\Bigr)\ge -\lambda_{\cc_{k},3}^{\max}+2\lambda_{\cc_{k},4}^{\min},
\\
\implies F_{\cc_k}(\alpha, \varepsilon_{k,0})\ge 1.0262.
\end{gather*}

\paragraph*{\normalsize{Case 2.2.3.6. $b_{k+4}=2, b_{k+5}\ge 1$}.}
We obtain
\begin{gather*}
0>\varepsilon_{k,0}\ge
q_k\Bigl(-\delta_{k+3}+2\delta_{k+4}-(a_{k+6}-1)\delta_{k+5}-a_{k+8}\delta_{k+7}-\ldots\Bigr)\ge
q_k\Bigl(-\delta_{k+3}+\delta_{k+4}+\delta_{k+5}\Bigr)\ge
-\lambda_{\cc_{k},3}^{\max}+\lambda_{\cc_{k},4}^{\min}+
\lambda_{\cc_{k},5}^{\min},
\\
0>\varepsilon_{k+3,0}\ge
q_{k+3}\Bigl(-2\delta_{k+4}+\delta_{k+5}-(a_{k+7}-1)\delta_{k+6}-a_{k+9}\delta_{k+8}-\ldots\Bigr)\ge
q_{k+3}\Bigl(-2\delta_{k+4}+\delta_{k+6}\Bigr)\ge
-2\lambda_{\cc_{k+3},1}^{\max}+\lambda_{\cc_{k+3},3}^{\min},
\\
\implies \max\bigl(F_{\cc_k}(\alpha, \varepsilon_{k,0}), F_{\cc_k}(\alpha, \varepsilon_{k,0})F_{\cc_{k+3}}(\alpha, \varepsilon_{k+3,0})\bigr)\ge 1.0522.
\end{gather*}

\paragraph*{\large{Case 2.2.4. $b_{k+3}=2$}.}
\ \\
In this case, we have $a_{k+4}\ge 2$. We obtain
\begin{gather*}
0>\varepsilon_{k,0}\ge
q_k\Bigl(-2\delta_{k+3}-a_{k+6}\delta_{k+5}-a_{k+8}\delta_{k+7}-\ldots\Bigr)\ge
q_k\Bigl(-2\delta_{k+3}-\delta_{k+4}\Bigr)\ge
-2\lambda_{\cc_{k},3}^{\max}-\lambda_{\cc_{k},4}^{\max},
\\
\lambda_{\cc_{k+3}}^{\max}+\lambda_{\cc_{k+3},1}^{\max}\ge\varepsilon_{k+3,1}>0,
\\
\implies \max\bigl(F_{\cc_k}(\alpha, \varepsilon_{k,0}), F_{\cc_k}(\alpha, \varepsilon_{k,0})F_{\cc_{k+3}}(\alpha, \varepsilon_{k+3,1})\bigr)\ge 1.14807.
\end{gather*}

\paragraph*{\large{Case 2.2.5. $b_{k+3}=3$}.}
\ \\
In this case, we have $a_{k+4}=3$. We obtain
\begin{gather*}
0>\varepsilon_{k,0}\ge
q_k\Bigl(-3\delta_{k+3}-a_{k+6}\delta_{k+5}-a_{k+8}\delta_{k+7}-\ldots\Bigr)\ge
q_k\Bigl(-3\delta_{k+3}-\delta_{k+4}\Bigr)\ge
-3\lambda_{\cc_{k},3}^{\max}-\lambda_{\cc_{k},4}^{\max},
\\
t\lambda_{\cc_{k+3}}^{\max}+\lambda_{\cc_{k+3},1}^{\max}\ge\varepsilon_{k+3,t}>(t-1)\lambda_{\cc_{k+3}}^{\min}, \quad t\ge 1,
\\
\implies\max\bigl(F_{\cc_k}(\alpha, \varepsilon_{k,0}), F_{\cc_k}(\alpha, \varepsilon_{k,0})F_{\cc_{k+3}}(\alpha, \varepsilon_{k+3,1})F_{\cc_{k+3}}(\alpha, \varepsilon_{k+3,2})\bigr)\ge 1.4541.
\end{gather*}\\

\section*{Appendix. Proof of Theorem \ref{thm_Gauss}.}

We will show that $\alpha\in S$ if and only if $T(\alpha)\in S$ which by iteration implies Theorem \ref{thm_Gauss}. By \cite{hauke_badly}, we have $S\subseteq\tilde{E}_6$ and $\tilde{E}_6$ is invariant under the Gauss map, hence we can assume without loss of generality that $\alpha$ is badly approximable. Actually we will prove the following stronger statement:

\begin{lem}\label{lemma_gauss}
    Let $P \in \N, \alpha\in E_P$ and let $N \in \N$ with Ostrowski expansion $N=\sum_{l=0}^n b_lq_l(\alpha)$. Writing $\alpha'=T(\alpha)$ and $N':=\sum_{l=1}^n b_lq_{l-1}(\alpha')$, we have
       \begin{equation}
\label{N_Nprime}    
P_N(\alpha)\asymp P_{N'}(\alpha'),
\end{equation}
with the implied constant only depending only on $P$.
\end{lem}

 Note that Lemma \ref{lemma_gauss} can be compared to estimates 
 in \cite[Section 4]{zag_conj} obtained for the case where $\alpha$ is not badly approximable, in which case $\frac{P_N(\alpha)}{P_{N'}(\alpha')}$ is not known to be bounded.
 Since $N\to N'$ is a bijection for $N>P$, Lemma \ref{lemma_gauss} implies Theorem \ref{thm_Gauss}.

 \begin{proof}[Proof of Lemma \ref{lemma_gauss}]
Recall that by Proposition \ref{prop_shifted}, we have

\[
P_N(\alpha) = \prod_{k=0}^{n}\prod_{t= 0}^{b_k-1} P_{q_k}\bigl(\alpha,\varepsilon_{k,t}(N)\bigr),\quad
P_{N'}(\alpha')=
\prod_{k=1}^{n}\prod_{t=0}^{b_k-1} P_{q_{k-1}(\alpha')}\bigl(\alpha',\varepsilon_{k-1,t}(N', \alpha')\bigr)
\]
and by Proposition \ref{limit_H} for any compact interval $I$,
\begin{equation}
\label{PkHk}
P_{q_k}(\alpha,\varepsilon) = H_k(\alpha,\varepsilon)\left(1 + \mathcal{O}\left(q_k^{-2/3}\log^{2/3}q_k\right)\right) + \mathcal{O}(q_k^{-2}),\quad \varepsilon \in I,
\end{equation}
with the implied constant only depending on $P$ and $I$. 

%
We recall that 
$$
H_k(\alpha,\varepsilon) =
2\pi \lvert \varepsilon + \lambda_k(\alpha) \rvert \prod_{n=1}^{\lfloor q_k/2\rfloor } h_{n,k}(\alpha,\varepsilon),
$$
where
\begin{equation*}
h_{n,k}(\alpha,\varepsilon) = \Bigg\lvert\Bigg(1 - \lambda_k\frac{\left\{n\cev{\alpha}_k\right\} - \frac{1}{2}}{n}\Bigg)^2 - \frac{\left(\varepsilon + \frac{\lambda_k}{2}\right)^2}{n^2}\Bigg\rvert.
\end{equation*}

Note that $q_{k-1}(\alpha')<q_k(\alpha)<(P+1)q_{k-1}(\alpha')$ and therefore $\forall \sigma\in\mathbb{R}$ all the quantities of the form $O(q_k(\alpha))^{\sigma}$ are $O(q_{k-1}(\alpha'))^{\sigma}$ and vice versa. Due to this fact, we will write $O(q_{k}^{\sigma})$ when it does not create ambiguity. We split the proof of \eqref{N_Nprime} into the following steps, where all the implied constants only depend on $P$:

\begin{itemize}
\item Step 1: For $k \geq 1$, $1 \leq n \leq \sqrt{q_k(\alpha)}, 0 \leq t \leq b_k -1$, we have
\begin{align}\label{convergence_eps}
 \varepsilon_{k,t}(N, \alpha)&=\varepsilon_{k-1,t}(N', \alpha')\bigl(1+O(1/q_k^2)\bigr),\\
\lambda_{k}(\alpha)&=\lambda_{k-1}(\alpha')\bigl(1+O(1/q_k^2)\bigr),\label{convergence_lambda}\\
\{n\cev{\alpha}_k\}&=\{n\cev{\alpha}'_{k-1}\}\bigl(1+O(1/q_k^{3/2})\bigr)\label{convergence_na}.\end{align}\\
\item Step 2:
For any $|\varepsilon|<1$,
\begin{equation}
\label{hk_large_factors}
\prod_{n=\lfloor\sqrt{q_k(\alpha)}\rfloor+1}^{\lfloor q_k(\alpha)/2\rfloor } h_{n,k}(\alpha,\varepsilon)=1+O\left(\frac{1}{q_k^{0.49}}\right), \quad 
\prod_{n=\lfloor\sqrt{q_k(\alpha)}\rfloor+1}^{\lfloor q_{k-1}(\alpha')/2\rfloor } h_{n,k-1}(\alpha',\varepsilon)=1+O\left(\frac{1}{q_k^{0.49}}\right).
\end{equation}
\\
\item Step 3:
\begin{equation}
\label{prod_before_sqrt}
\prod\limits_{n=1}^{\lfloor\sqrt{q_k(\alpha)}\rfloor}h_{n,k}(\alpha,\varepsilon)=\prod\limits_{n=1}^{\lfloor\sqrt{q_k(\alpha)}\rfloor}h_{n,k-1}(\alpha',\varepsilon)\left(1+O\left(\frac{1}{q_k^{1.49}}\right)\right). 
\end{equation}
\item Step 4: The functions $(H_k(\alpha,\varepsilon))_{k \in \N}$ are uniformly Lipschitz continuous on compact intervals $I$, with the Lipschitz constant only depending on $I$ and $P$.
\item Step 5:
\begin{equation}
\label{step4equiv}
H_k(\varepsilon_{k,t}(N, \alpha))=H_{k-1}(\varepsilon_{k-1,t}(N', \alpha'))\left(1+O\left(\frac{1}{q_k^{0.49}}\right)\right).
\end{equation}
\end{itemize}
Now we deduce \eqref{N_Nprime}. In the view of (\ref{PkHk}), one can write (\ref{step4equiv}) as
$$
P_{q_k}\bigl(\alpha,\varepsilon_{k,t}(N)\bigr)=
P_{q_{k-1}(\alpha')}\bigl(\alpha',\varepsilon_{k-1,t}(N', \alpha')\bigr)\left(1+O\left(\frac{1}{q_k^{0.49}}\right)\right).
$$
Note that all the Ostrowski digits $b_k$ do not exceed $P$, therefore
$$
P_N(\alpha)\asymp P_{N'}(\alpha')\prod_{k=1}^{n}
\left(1+O\left(\frac{1}{q_k^{0.49}}\right)\right)\asymp P_{N'}(\alpha').
$$

We start the proof of lemma with proving Step 1: Note that \eqref{convergence_eps} follows from the definition of $N'$ and the fact that 
$$
q_k(\alpha)\delta_{k+t}(\alpha)=q_{k-1}(\alpha')\delta_{k+t-1}(\alpha')\left(1+O\left(\frac{1}{q_{k}^2}\right)\right).
$$ 
\eqref{convergence_lambda} follows from basic properties of continued fractions, so we are left to prove \eqref{convergence_na}.
For this to hold, it suffices to show that $\lfloor n\cev{\alpha}_k\rfloor=\lfloor n\cev{\alpha}'_{k-1}\rfloor$. 
Note that 
$$
|n\cev{\alpha}_k-n\cev{\alpha}'_{k-1}|\ll \frac{1}{q_k^{3/2}},
$$
so if there is an integer between $n\cev{\alpha}_k$ and $n\cev{\alpha}'_{k-1}$ then $\|n\cev{\alpha}_k\|\ll 1/q_k^{3/2}$ which is impossible since $n\le\sqrt{q_k(\alpha)}$ and $\cev{\alpha}_k$ has partial quotients bounded by $P$.\\

Proof of Step 2: 
Using the Taylor estimate $\exp(x)=(1+x)(1+O(x^2))$, we obtain
\begin{equation}
\label{prod_after_sqrt}
\prod_{n=\lfloor\sqrt{q_k(\alpha)}\rfloor+1}^{\lfloor q_k(\alpha)/2\rfloor } h_{n,k}(\beta,\varepsilon)=\exp\left(2\sum\limits_{n=\lfloor\sqrt{q_k(\alpha)}\rfloor+1}^{\lfloor q_k(\alpha)/2\rfloor}\frac{1/2-\{n\cev{\alpha_k}\}}{n}+\frac{O(1)}{n^2}\right)\left(1+O\left(\frac{1}{q_k(\alpha)}\right)\right).
\end{equation}

The first estimate of \eqref{hk_large_factors} now follows by partial summation and (\ref{sum_from_Tplus1}) in the same way as in Lemma \ref{taylor_lem}, the second estimate in \eqref{hk_large_factors} follows by using $q_{k-1}(\alpha') \asymp q_k(\alpha)$ and identical arguments.\\

Proof of Step 3:
From the estimates (\ref{convergence_lambda}) and (\ref{convergence_na}) for any $k \ge 1$, $1\le n\le \sqrt{q_k(\alpha)}$ we have 
\begin{equation}
h_{n,k}(\alpha,\varepsilon) =h_{n,k-1}(\alpha',\varepsilon)\left(1+O\left(\frac{1}{nq_k^{3/2}}\right)\right).
\end{equation}
Multiplying the equality above for $n=1,2,\ldots,\lfloor \sqrt{q_k}\rfloor$, we obtain (\ref{prod_before_sqrt}).\\

Proof of Step 4:
Note that $H_k(\alpha,\varepsilon)$ is everywhere differentiable except at its zeros.
Since $H_k$ is non-negative, it suffices to show the Lipschitz property between two such zeros,
thus without loss of generality, $I$ is differentiable in its interior.
Hence we can bound the derivative by

\[
\lvert H_k(\alpha,\varepsilon)'\rvert  \leq 2\pi \prod_{\substack{n = 1}}^{q_k/2}h_{n,k}(\varepsilon)
 + 2\pi \lvert \varepsilon + \lambda_k\rvert \lvert 2\varepsilon + \lambda_k\rvert \sum_{m \leq q_k/2} \frac{1}{m^2} \prod_{\substack{n = 1\\n \neq m}}^{q_k/2}h_{n,k}(\varepsilon).
\]

Note that $\prod\limits_{\substack{n = 1\\n \neq m}}^{q_k/2}h_{n,k}(\varepsilon) \ll_{I,P} 1$ uniformly in $m$ and $k$
and the same holds for $\prod\limits_{\substack{n = 1}}^{q_k/2}h_{n,k}(\varepsilon)$,
facts that can be proven analogously to Step 2. Since also $\lambda_k$ is uniformly bounded, 
it follows that $\lvert H_k(\alpha,\varepsilon)'\rvert \ll_{I,P} 1$
which implies $H_k$ being uniformly Lipschitz continuous.\\

Proof of Step 5:
Due to (\ref{hk_large_factors}) and (\ref{prod_before_sqrt}) for any $|\varepsilon|<1$ one has
\begin{equation}
\begin{split}
H_k(\alpha,\varepsilon) &=
2\pi \lvert \varepsilon + \lambda_k(\alpha) \rvert \prod_{n=1}^{\lfloor q_k(\alpha)/2\rfloor } h_{n,k}(\alpha,\varepsilon)=
2\pi \lvert \varepsilon + \lambda_{k-1}(\alpha') \rvert
\prod_{n=1}^{\lfloor q_{k-1}(\alpha')/2\rfloor } h_{n,k-1}(\alpha',\varepsilon)\left(1+O\left(\frac{1}{q_k^{0.49}}\right)\right)\\&=H_{k-1}(\alpha',\varepsilon)\left(1+O\left(\frac{1}{q_k^{0.49}}\right)\right).
\end{split}    
\end{equation}
As the functions $H_k(\alpha,\varepsilon)$ are uniformly Lipschitz-continuous, taking into account (\ref{convergence_eps}), we obtain (\ref{step4equiv}).
 \end{proof}

\subsection*{Acknowledgements} 
We would like to thank Daniel El-Baz for making us aware of Ostrowski's estimate (\ref{Ostr_est}). We are thankful to Christoph Aistleitner and Bence Borda for helpful comments on
an earlier version of this manuscript. DG was supported by the Austrian Science Fund (FWF) Projects I-5554 and P-34763. 
MH was supported by the Austrian Science Fund (FWF) Project P-35322 and by the EPSRC grant EP/X030784/1.

\end{document}